\documentclass[11pt, letterpaper]{amsart}
\usepackage[left=1in,right=1in,bottom=1.5in,top=1.5in]{geometry}
\usepackage{amsmath}	
\usepackage{fancyhdr}	
\usepackage{graphicx}	
\usepackage{cancel}		
\usepackage{amsfonts}
\usepackage{amsthm}
\usepackage{amssymb}
\usepackage{graphicx}
\usepackage[font=small,labelfont=bf]{caption}
\usepackage{epstopdf}
\usepackage[pdfpagelabels,hyperindex]{hyperref}
\usepackage{xcolor}
\usepackage{amsthm}
\usepackage{float}
\usepackage{listings}
\usepackage{longtable}
\usepackage{mathrsfs}
\usepackage{setspace}
\usepackage{mathtools}
\DeclarePairedDelimiter{\ceil}{\lceil}{\rceil}

\numberwithin{equation}{section}

\newtheorem{theorem}{Theorem}[section]
\newtheorem{lemma}[theorem]{Lemma}
\newtheorem{corollary}[theorem]{Corollary}
\newtheorem{proposition}[theorem]{Proposition}
\newtheorem{definition}[theorem]{Definition}

\newtheorem{remark}[theorem]{Remark}

\newcommand{\abs}[1]{\lvert #1 \rvert}

\author{Alastair N. Fletcher and Allyson M. Hahn}

\title{Geometric Function Theory on Uniformly Quasiconformally Homogeneous Domains}

\begin{document}
\begin{abstract}
Uniformly quasiconformally homogeneous domains in $\mathbb{R}^n$ carry a transitive collection of $K$-quasiconformal maps for a fixed $K\geq 1.$ In this paper, we study two questions in this setting. The first is to show that quasiconformality and quasisymmetry with respect to the quasihyperbolic metric are equivalent. The second is to study normal quasiregular maps from such a domain into $S^n$ or $\mathbb{R}^n$ and show they enjoy geometric properties such as a uniform H\"{o}lder condition.
\end{abstract}
\maketitle

\section{Introduction}
\subsection{Uniformly Quasiconformally Homogeneous Domains}In complex analysis, geometric function theory provides insight as to the geometric properties that holomorphic functions enjoy. If a holomorphic function maps between two simply connected domains in $\mathbb{C},$ by the Riemann Mapping Theorem, we can lift via conformal bijections to work in the unit disk. In particular, we are able to take advantage of hyperbolic geometry and the fact that holomorphic functions do not increase hyperbolic distance. A common model is the Poincar\'{e} Disk, where we view the unit disk, denoted $\mathbb{D},$ as the hyperbolic plane and equip it with the hyperbolic metric. In this setting, we gain the important property that the group of automorphisms of $\mathbb{D}$ form a conformal, transitive collection of hyperbolic isometries. These characteristics and more are outlined by Beardon and Minda in \cite{beardon2011hyperbolic}.  

Since quasiregular mappings are a natural generalization of holomorphic functions to higher real dimensions, we can generalize geometric function theory to the setting of $\mathbb{R}^n,$ for $n\geq 2.$ Hyperbolic geometry generalizes nicely, as we can use the $n$-dimensional unit ball, denoted $\mathbb{B}^n,$ equipped with the hyperbolic metric as our model. With this model, we obtain the property that the automorphisms of $\mathbb{B}^n$ form a conformal, transitive collection of hyperbolic isometries. However, without the power of the Riemann Mapping Theorem in higher dimensions, we are limited to maps whose domain and codomain are contained in $\mathbb{B}^n$ or half-spaces $\mathbb{H}^n.$ Generalizing this property to other domains in $\mathbb{R}^n$ equates to defining an arbitrary subdomain of $\mathbb{R}^n$ which has a transitive collection of mappings with a global geometric property in some ``hyperbolic-like" metric.

Towards that end, in the highly influential paper \cite{gehring1976quasiconformally}, Gehring and Palka introduced the following definition:

\begin{definition}
    Let $X\subset \mathbb{R}^n$ be a domain and let $K\geq 1.$ We call $X$ a \textit{uniformly $K$-quasiconformally homogeneous} domain, or \textit{uniformly $K$-QCH domain}, if there exists a collection $G$ of $K$-quasiconformal mappings which act transitively on $X.$ In other words, for any $x,y\in X$ there exists $g\in G$ such that $g(x) = y.$
\end{definition}

Note that $G$ above is not guaranteed to be a group, just a collection. Liouville's Theorem tells us the class of M\"{o}bius maps are the only conformal maps in $\mathbb{R}^n$ for $n\geq 3,$ so requiring quasiconformal maps as our transitive collection is the appropriate generalization. It was shown in \cite[ Lemma $3.2$]{gehring1976quasiconformally} that every proper domain admits a collection of quasiconformal maps which act transitively, but such a collection is not guaranteed to have a uniform bound on the maximal dilatations. Requiring all maps in the collection to be $K$-quasiconformal for $K\geq 1,$ as in the definition above, upgrades our domain, allowing for more control over geometric properties. 
 
Obviously $\mathbb{B}^n$ is a uniformly $1$-QCH domain, as $\mathbb{B}^n$ carries the transitive group of M\"{o}bius self maps, but let us discuss some non-trivial examples of uniformly QCH domains. Gehring and Palka proved that there exists a uniformly $K$-QCH domain whose complement is a Cantor set \cite[Example $4.4$]{gehring1976quasiconformally}. They also have shown that for each $n\geq 2$ there is a domain $D\subset \mathbb{R}^n$ which is uniformly quasiconformally homogeneous such that the complement of $D$ has at least two non degenerate components \cite[Example $4.6$]{gehring1976quasiconformally}. The authors of \cite{bonfert2005quasiconformal} investigate the properties hyperbolic manifolds require in order to be quasiconformally homogeneous domains, proving that if $n\geq 3,$ a hyperbolic $n$-manifold
is uniformly quasiconformally homogeneous if and only if it is a regular cover of a closed
hyperbolic orbifold. This work is continued in \cite{bonfert2007quasiconformal}, where it is proven that any closed hyperbolic surface admitting a conformal automorphism with ``many" fixed points is uniformly quasiconformally homogeneous, with constant uniformly bounded away from 1. Other than acquiring examples of uniformly QCH domains, two other lucrative research areas in this field are determining bounds for $K$ and proving various results for domains that are assumed to be uniformly $K$-QCH (\cite{bonfert2010ambient}, \cite{gehring1976quasiconformally}, and \cite{gong2009aspects}).

In addition to defining uniformly QCH domains, Gehring and Palka defined the quasihyperbolic metric in \cite{gehring1976quasiconformally}, which generalizes properties of the hyperbolic metric but is for a proper subdomain $X$ of $\mathbb{R}^n$ and is denoted $k_X.$ This is a conformal metric with density $\frac{1}{d(X,\partial X)}.$ Many ``hyperbolic-like" properties are preserved by the quasihyperbolic metric \cite{hariri2020conformally}, but we are not guaranteed $X$ has a large collection of conformal, transitive, quasihyperbolic isometries. Requiring $X$ to also be a uniformly $K$-QCH domain gives us a domain with a large collection of transitive,  $K$-quasiconformal maps. Moreover, the $K$-quasiconformal maps which arise from the domain have a global geometric property on $X$ \cite{gehring1979uniform}. Thus, uniformly QCH domains equipped with the quasihyperbolic metric correctly generalize $\mathbb{B}^n$ equipped with the hyperbolic metric and the transitive collection of conformal, hyperbolic isometries that arise from it.

Gehring and Palka's original intention for defining uniformly QCH domains and equipping them with their quasihyperbolic metric was in hopes of gaining a new characterization for domains which are quasiconformally equivalent to the unit ball. In other words, they desired to find domains for which we could have a truly generalized Riemann Mapping Theorem. This vein of research is still pursued by some (\cite{bonfert2007quasiconformal},\cite{bonfert2010ambient}, and \cite{bonfert2011teichmuller}) but this will not be our area of focus. In this paper, we aim to study the geometric properties of quasiregular mappings on uniformly QCH domains equipped with the quasihyperbolic metric.

\subsection{Statement of Results}
 We investigate the relationship between quasisymmetric and quasiconformal mappings. While quasiconformality is a local property that concerns the distortion of small circles, quasisymmetry is a global three point condition for maps between metric spaces which preserves the relative size of sets. See Section $2$ for full definitions.

Since the 1960's, we have known that quasiconformal and quasisymmetric mappings are equivalent when the domain and codomain are $\mathbb{R}^n$ for $n\geq 2$ \cite{gehring1962rings}. The same equivalence does not hold for other subdomains in general. For example, in \cite[p. 135]{hubbard2016teichmuller}, it was shown that a conformal map from the unit disk to the slit disk in the complex plane is $1$-quasiconformal, but is not quasisymmetric in the Euclidean metric. However, Tyson \cite{tyson1998quasiconformality} proved a quasisymmetric map is quasiconformal if the metric spaces of interest are locally compact and connected Ahlfors $Q$-regular spaces with Hausdorff dimension greater than $1.$ 

As for the converse question, much has been done to prove quasiconformal maps are quasisymmetric in metric spaces with properties similar to that of Euclidean space. In \cite{heinonen2001lectures}, Heinonen proved a quasiconformal map from a bounded uniform domain onto a bounded, linearly locally connected domain in $\mathbb{R}^n$ is quasisymmetric. Heinonen also states that analogous results can be proven using the interior metric for the domain. Moreover, Heinonen and Kosela \cite{heinonen1998quasiconformal} proved that if $X$ and $Y$ are Ahlfors $Q$-regular metric spaces with $Q>1,$ $X$ a Loewner space, $Y$ a linearly locally connected space, and $f:X\rightarrow Y$ quasiconformal, then $f:X\rightarrow Y$ is quasisymmetric. These results in the Euclidean setting are useful. However, Ackermann and the first author noticed that although in $Q$-regular metric spaces balls of radius $r$ are comparable to $r^Q,$ this is not true in hyperbolic space \cite{ackermann2021quasiconformality}. The size of the balls will grow exponentially, so the argument made in \cite{heinonen1998quasiconformal} does not hold in the hyperbolic setting. In \cite{ackermann2021quasiconformality}, they fill this gap by proving that a quasiconformal map $f:\mathbb{B}^n\rightarrow \mathbb{B}^n$ is in fact quasisymmetric, where $\mathbb{B}^n$ is equipped with the hyperbolic metric. Moreover, the notions of quasisymmetry and quasiconformality coincide in the hyperbolic setting.

We first aim to generalize Ackermann and the first author's work. By requiring $X$ to be a uniformly $K$-QCH domain equipped with the quasihyperbolic metric we gain a transitive collection of $K$-quasiconformal mappings with a desirable geometric property. This is a reasonable choice, as such domains properly generalize the transitive collection of conformal hyperbolic isometries that we are afforded in the hyperbolic setting. 
\begin{theorem}
    \label{qcisqs}
    Let $X,Y$ be proper subdomains of $\mathbb{R}^n$ and suppose $X$ is a uniformly $M$-QCH domain. Then, $f:(X,k_X)\rightarrow (Y,k_Y)$ is $K$-quasiconformal if and only if $f$ is $\eta$-quasisymmetric, in the respective quasihyperbolic metrics.
\end{theorem}

We then shift our attention to studying the geometric properties of particular classes of mappings. Normal meromorphic functions were introduced by Lehto and Virtanen in \cite{lehto1957boundary}, and Bloch functions were investigated by Pommerenke in \cite{pommerenke1970bloch}. Recall that a meromorphic function $f:\mathbb{D}\rightarrow \overline{\mathbb{C}}$ is \textit{normal} if $\{f\circ A: A\in G\}$ is a normal family, where $G$ is the set of M\"{o}bius maps on $\mathbb{D},$ and a holomorphic function $g:\mathbb{D}\rightarrow \mathbb{C}$ is a \textit{Bloch} function if the family $\{g(A(z))-g(A(0)): A\in G\}$ is a normal family, where $G$ is again the set of M\"{o}bius maps on $\mathbb{D}.$ These classes of mappings have been studied at length in the planar setting.

Such maps have been generalized to higher dimensions. The generalization for normal meromorphic functions was established by Vuorinen in \cite{vuorinen2006conformal}. The generalization of Bloch functions was systematically explored for the first time by the first author and Nicks in \cite{fletcher2024normal}. The first author and Nicks define \textit{normal quasiregular mappings} which are quasiregular maps $f:X\rightarrow S^n$ such that $\{f\circ A:A\in G\}$ is a normal family, where $X\subset S^n$ is a metric space arising from a conformal metric and $G$ is a collection of conformal isometries of $X.$ If the range is instead $\mathbb{R}^n$ then $f$ is normal if the family $\{f(A(x))-f(A(x_0)): A\in G\}$ is normal for some $x_0\in X.$ These easily compare to the definitions of normal meromorphic and Bloch functions. However, there are only a few known examples of domains $X$ with a transitive collection of isometries: $X = \mathbb{B}^n$ and $G$ the collection M\"{o}bius maps or $X = \mathbb{R}^n$ and $G$ the collection of translations. Consider the following definition.

\begin{definition}
    If $X,Y$ subdomains of $S^n,$ then denote by $\mathcal{Q}_K(X,Y)$ the subset of $C(X,Y)$ consisting of all $K$-quasiregular mappings from $X$ to $Y$ for $K\geq 1.$
\end{definition}

Then, we can generalize the definition of a normal quasiregular map, by loosening the restriction that $G$ be a transitive collection of isometries. We accomplish this by letting our domain be a uniformly quasiconformally homogeneous domain.

\begin{definition}\label{normalqrmapdef}
Let $n\geq 2.$ Let $(X,d_X)$ be a metric space arising from a conformal metric, and suppose $X$ is a uniformly $M$-QCH  domain. Further, let $G$ be the transitive collection of orientation-preserving $M$-quasiconformal mappings $A:X\rightarrow X$ arising from $X.$ 
\begin{itemize}
    \item[(i)] We say that a $K$-quasiregular map $f:X\rightarrow S^n$ is a \textit{normal quasiregular map} into $S^n$ if the family
    \begin{align*}
        \mathcal{F} & = \{f\circ A:A\in G\} \subset \mathcal{Q}_{KM}(X,S^n)
    \end{align*}
    is a normal family.

    \item[(ii)]We say that a $K$-quasiregular map $f:X\rightarrow \mathbb{R}^n$ is a \textit{normal quasiregular map} into $\mathbb{R}^n$ if the family
    \begin{align*}
        \mathcal{F} & = \{f(A(x))-f(A(x_0)):A\in G\} \subset \mathcal{Q}_{KM}(X,\mathbb{R}^n)
    \end{align*}
    is normal for some $x_0\in X.$\\
\end{itemize}
\end{definition} 

For the remainder of the paper, when we refer to normal quasiregular mappings, we are discussing this generalized version in Definition \ref{normalqrmapdef}. We are then able to prove geometric results about such mappings analogous to those acquired in \cite{fletcher2024normal}. The first of these refers to uniform continuity.

\begin{theorem}
\label{uniformconttheorem}
Let $X\subset S^n$ be a proper subdomain equipped with the quasihyperbolic metric. Furthermore, suppose $X$ is a uniformly $M$-QCH domain with $G$ as the transitive collection of orientation-preserving $M$-quasiconformal mappings $A:X\rightarrow X$ arising from $X.$ Then,

\begin{itemize}
    \item[(i)]$f:X\rightarrow S^n$ is a normal quasiregular mapping if and only if $f$ is uniformly continuous with respect to $k_X$ and $\sigma.$
    
    \item[(ii)]$f:X\rightarrow \mathbb{R}^n$ is a normal quasiregular mapping if and only if $f$ is uniformly continuous with respect to $k_X$ and the Euclidean distance.\\ 
\end{itemize}
\end{theorem}

If in the definition above, we replace $S^n$ or $\mathbb{R}^n$ with a subdomain $Y\subset S^n$ equipped with the quasihyperbolic metric, and $Y$ omits enough points (Rickman's constant), then $f$ is automatically a normal quasiregular mapping by \cite[Thm IV.2.1]{rickman2012quasiregular}. Thus, only studying when the range is $S^n$ or $\mathbb{R}^n$ is reasonable.

Next, we prove a global H\"{o}lder continuity result for normal quasiregular maps when the range is $S^n$ equipped with the spherical metric.
\begin{theorem}
\label{holderSn}
 Let $X\subset S^n$ be a proper subdomain equipped with the quasihyperbolic metric. Furthermore, suppose $X$ is a uniformly $M$-QCH domain with $G$ as the transitive collection of orientation-preserving $M-$quasiconformal mappings $A:X\rightarrow X$ arising from $X.$ Let $f:X\rightarrow S^n$ be a $K-$quasiregular mapping and let $\beta = (KM^2)^{1/(1-n)}.$ Then $f$ is normal if and only if $f$ is globally $\beta$-H\"{o}lder, that is, there exists $C_0>0$ such that 
\begin{align}
    \sigma(f(x),f(y)) & \leq C_0k_X(x,y)^\beta,\label{4.2.1}
\end{align}
for all $x,y\in X.$
\end{theorem}
A similar result can be shown for the normal quasiregular maps when the range is $\mathbb{R}^n.$
\begin{theorem}
\label{holderRn}
Let $X\subset S^n$ be a proper subdomain equipped with the quasihyperbolic metric. Suppose $X$ is a uniformly $M$-QCH domain with $G$ as the transitive collection of orientation-preserving $M$-quasiconformal mappings $A:X\rightarrow X$ arising from $X.$ Let $f:X\rightarrow \mathbb{R}^n$ be a $K$-quasiregular mapping and let $\beta = (KM^2)^{1/(1-n)}.$ Then, $f$ is normal if and only if there exists $C_0>0$ such that
\begin{align}
    \abs{f(x)-f(y)} & \leq C_0\max\{k_X(x,y),k_X(x,y)^{\beta}\},\label{4.5}
\end{align}
for all $x,y\in X.$
\end{theorem}

 While our focus is on uncovering geometric properties of normal quasiregular mappings, it is useful to note some examples and non-examples of such mappings. Recall in \cite[Thm 9]{lehto1957boundary}, Lehto and Virtanen prove a planar meromorphic function cannot be normal in any neighborhood of an isolated essential singularity. For example, there do not exist any normal meromorphic maps in $\mathbb{B}^2\setminus \{0\}$ equipped with the hyperbolic metric for which $0$ is an essential singularity. This result was then generalized to higher dimensions by Heinonen and Rossi in \cite[Thm 2.3]{heinonen1990lindelof}. They show if $f: (\mathbb{B}^n\setminus \{0\},d_\delta)\rightarrow (S^n,\sigma)$ has an essential singularity at $0,$ then $f$ cannot be a normal quasiregular map, where the metric $d_\delta$ is defined by
\begin{align*}
    d_\delta(x,y) = \inf\int_\gamma \delta(s) \abs{ds},
\end{align*}
the infimum is taken over all paths $\gamma$ joining $x$ and $y$ in $\mathbb{B}^n\setminus\{0\},$ and the density $\delta$ is required to be a positive continuous function in $\mathbb{B}^n\setminus \{0\}$ with $\delta(x) = o(\abs{x}^{-1}),$ as $x\rightarrow 0. $ The construction of $d_\delta$ is meant to mimic the behavior of the hyperbolic metric on $\mathbb{B}^2\setminus\{0\}$ near the singularity. However, the quasihyperbolic metric is not comparable to $d_\delta$ on $\mathbb{B}^n\setminus\{0\},$ as its density is $O(\abs{x}^{-1})$ as $x\rightarrow 0.$ Thus, it would be interesting to see whether a normal quasiregular map from a punctured neighborhood $U\setminus\{x_0\}$ equipped with the quasihyperbolic metric to $S^n$ equipped with the spherical metric requires $x_0$ be a removable singularity or a pole. We leave this as an open question.

The paper is organized as follows. In Section $2,$ we provide necessary preliminary information. Sections $3$ and $4$ house our main results. Section $3$ focuses on investigating the relationship between quasisymmetric and quasiconformal mappings on $K$-QCH domains, while Section $4$ centers on geometric properties of normal quasiregular mappings on $K$-QCH domains.

\section{Preliminaries}

\subsection{Quasihyperbolic Metric}
Let $X\subset S^n$ be a domain. We say the form $\lambda_X(x)\abs{dx}$ is a \textit{conformal metric} on $X$ if $\lambda_X$ is a strictly positive continuous function. The conformal metric then induces a distance function given by
\begin{align*}
    d_X(x,y) & = \inf\int_\gamma \lambda_X(x)\abs{dx},
\end{align*}
where the infimum is over all paths joining $x$ to $y$ in $X.$ For example, the Euclidean distance function on $\mathbb{R}^n$ arises from the conformal metric $\lambda_X(x) = 1$ and the hyperbolic distance $\rho$ on $\mathbb{B}^n$ arises from $\lambda_X(x) = 2(1-\abs{x}^2)^{-1}.$ A final example to note is if we view the unit sphere $S^n$ in $\mathbb{R}^{n+1}$ as $\mathbb{R}^n\cup \{\infty\},$ then we define the spherical distance function on $S^n$ by
\begin{align*}
    \sigma(x,y) & = \inf \int_\gamma \frac{2}{1+\abs{x}^2}\abs{dx},
\end{align*}
where $x,y\in S^n$ and the infimum is over all paths in $S^n$ from $x$ to $y.$\\
Throughout this paper we will primarily focus on one metric in particular. Given a proper subdomain $X\subset\mathbb{R}^n$ the quasihyperbolic distance function arises from the conformal metric
\begin{align*}
    \lambda_X(x) & = \frac{1}{d(x,\partial X)},
\end{align*}
where $d$ here denotes Euclidean distance. We denote the quasihyperbolic distance function on $X$ by $k_X,$ and for $x,y\in X$ it is given by
\begin{align*}
    k_X(x,y) & = \inf \int_\gamma \frac{1}{d(x,\partial X)}\abs{dx},
\end{align*}
where again the infimum is taken over all paths in $X$ joining $x$ to $y.$ The quasihyperbolic metric was introduced by Gehring and Palka in the highly influential paper \cite{gehring1976quasiconformally}.  Let us state a few important, non-trivial results for domains equipped with the quasihyperbolic metric.

\begin{proposition}[\cite{gehring1979uniform}, Lemma 1]
    \label{geodesicspace}
    Let $X$ be a proper subdomain of $\mathbb{R}^n$. For every $x,y\in X$ there exists a quasi-hyperbolic geodesic $\gamma$ with endpoints $x$ and $y.$
\end{proposition}

\begin{proposition}[\cite{gehring1976quasiconformally}, Corollary 2.2]
    \label{qhcomplete}
    If $X\subset \mathbb{R}^n$ proper subdomain of $\mathbb{R}^n$, then $k_X$ is a complete metric on $D,$ which determines the usual topology there.
\end{proposition}

Before we conclude our discussion on metric spaces, consider the following restriction we can place on a metric space.

\begin{definition}
    A \textit{path} in $\mathbb{R}^n$ is a continuous mapping $\gamma: [a,b]\rightarrow \mathbb{R}^n.$ If $a = t_0 \leq t_1\leq \ldots\leq t_m = b$ is a partition of $[a,b].$ The supremum of the sums $\sum_{i=1}^m \abs{\gamma(t_i)-\gamma(t_{i-1})}$ over all partitions is called the $\textit{length}$ of $\gamma,$ denoted $\ell(\gamma).$ If $\ell(\gamma)<\infty$ then $\gamma$ is called a \textit{rectifiable} path. 
\end{definition}

\begin{definition}
    \label{quasiconvex}
    Let $(X,d_X)$ be a metric space. Then $X$ is \textit{$c$-quasiconvex}, for $c\geq 1,$ if each pair of points $x,y\in X$ can be joined by a rectifiable path $\gamma$ such that the length of the path $\ell(\gamma)\leq cd_X(x,y).$ 
\end{definition}
\begin{remark}
    \label{quasihypisquasconvex}
    If $(X,d_X)$ is a geodesic metric space then for all $x,y$ we can find a geodesic $\gamma$ with endpoints $x$ and $y,$ and $\ell(\gamma) = d_X(x,y).$ So, $(X,d_X)$ is also $1$-quasiconvex. In particular, Proposition \ref{geodesicspace} implies $(X,k_X)$ is $1$-quasiconvex.
\end{remark}

\subsection{Quasiregular Mappings}

Quasiregular mappings in $\mathbb{R}^n$ are a natural generalization of holomorphic functions in the plane. We provide an outline of necessary definitions for quasiregular and quasiconformal mappings below, but refer to Rickman's monograph \cite{rickman2012quasiregular} and V\"{a}is\"{a}l\"{a}'s lectures \cite{vaisala2006lectures} for more details. We begin with the definition of a quasiregular map and an equivalent characterization of a quasiconformal map. Then, we provide a few tools and results which will be of use later on. 

\begin{definition}(Analytic Definition)
    Let $U\subset\mathbb{R}^n$ be a domain. A map $f:U\rightarrow \mathbb{R}^n$ is said to be $K$-\textit{quasiregular} if $f$ is $\mathrm{ACL}^n$ and if there exists a constant $K\geq 1$ such that
    \begin{align*}
        \abs{f'(x)}^n\leq KJ_f(x),
    \end{align*}
    almost everywhere in $U.$ The smallest $K$ for which this holds is called the \textit{outer dilatation} of $f,$ denoted $K_O(f).$ Moreover, if $f$ is $K$-quasiregular, then
    \begin{align*}
        J_f(x) \leq K'\min_{\abs{h}=1}\abs{f'(x)h}^n,
    \end{align*}
    where $K'$ is the smallest value for which the inequality holds and is called the \textit{inner dilatation}, denoted $K_I(f).$ The \textit{maximal dilatation} is $K(f) = \max\{K_O(f),K_I(f)\}.$ If $f$ is quasiregular and also injective, then $f$ is called \textit{quasiconformal}.
\end{definition}

As quasiregular mappings are not required to be injective, it is useful to determine whether they are injective on ``small scales". The \textit{local index} of a quasiregular map $f$ at $x_0\in \mathbb{R}^n$ is
\begin{align*}
    i(x_0,f) & = \inf_U \sup_x \mathrm{card}(f^{-1}(x)\cap U),
\end{align*}
where the infimum is taken over all neighborhoods $U$ of $x_0.$ If $i(x_0,f) = 1$ then $f$ is called \textit{locally injective at} $x_0.$ A point $x_0$ where $i(x_0,f)>1$ is called a \textit{branch point} of $f.$ 

Since we will be working with metric spaces, considering the metric definition of quasiconformal mappings will be of use to us. Let us define some preliminary notation. Let $f:(X,d_X)\rightarrow (Y,d_Y)$ be a homeomorphism, $x_0\in X$ and $r>0$ be such that $\overline{B(x_0,r)}\subset X.$ Then, we define the \textit{linear dilatations}    \begin{align*}
        L_f(x_0,r) & = \sup\{d_Y(f(x),f(x_0)):d_X(x,x_0)=r\},\\
        \ell_f(x_0,r) & = \inf\{d_Y(f(x),f(x_0)):d_X(x,x_0)=r\}.
    \end{align*}
We see that $L_f(x_0,r)$ is measuring the maximal stretch of $\partial B_X(x_0,r)$ by $f$ and $\ell_f(x_0,r)$ measures the minimal stretch. This idea is key to proving one of our main results later on. However, for now, we can use this notion of maximal and minimal stretch to formulate an equivalent definition of a quasiconformal map.
\begin{definition} (Metric Definition)
    With the notation as above, the \textit{linear dilatation} of $f$ at $x_0$ is 
    \begin{align*}
        H_f(x_0) & = \limsup_{r\rightarrow 0} \frac{L_f(x_0,r)}{\ell_f(x_0,r)}.
    \end{align*}
    Then, $f$ is $K$-quasiconformal if and only if $\|H_f\|_\infty\leq K$ for some $K\geq 1.$
\end{definition}\label{metricspacedefofqc}
The analytic and metric definitions of quasiconformal maps were proven to be equivalent in \cite[Corollary 4]{gehring1962rings}. However, there is yet another characterization of quasiconformal mappings which uses modulus of path families. Although this interpretation of quasiconformal maps will not be of use to us, modulus of path families serve as a crucial tool in proving one of our main results. Thus, we take a moment to define the modulus of path families and note some important properties

A \textit{path family}, denoted $\Gamma,$ is a collection of paths $\gamma$ in $\mathbb{R}^n.$ Let $F(\Gamma)$ be the set of all non-negative Borel functions $\rho:\mathbb{R}^n\rightarrow \mathbb{R}$ such that
\begin{align*}
    \int_\gamma \rho\hspace{1mm} dm & \geq 1
\end{align*}
for every rectifiable curve $\gamma \in \Gamma,$ where $m$ is the Lebesgue measure. Now, we can properly define the modulus of a path family below.
\begin{definition}
    The modulus of $\Gamma,$ denoted $M(\Gamma)$ is defined by
    \begin{align*}
        \inf_{\rho\in F(\Gamma)} \int_{\mathbb{R}^n} \rho^n dm,
    \end{align*}
    where $m$ is Lebesgue measure. If $F(\Gamma) = \emptyset,$ then we set $M(\Gamma) = \infty.$ 
\end{definition}
Next, we acknowledge a few properties of the modulus of path families that will be important to us. We note that this does not cover all information on modulus of path families, so please consult \cite[Chapter II]{vaisala2006lectures} for further information.  First, if every path in the family $\Gamma_1$ is a subpath of the family $\Gamma_2,$ then $M(\Gamma_2)\leq M(\Gamma_1).$ In other words, shorter paths will result in a larger modulus. To see this in action, let $0<a<b<\infty$ and the ring domain $A = \{x\in \mathbb{R}^n:a<\abs{x}<b\}.$ Also, set $\Gamma_A$ to be the family of paths joining $\{x:\abs{x}=a\}$ and $\{x:\abs{x}=b\}$ in $A.$ Then,
\begin{align*}
    M(\Gamma_A) & = \omega_{n-1}\left(\log\left(\frac{b}{a}\right)\right)^{1-n},
\end{align*}
where $\omega_{n-1}$ is the $(n-1)$-dimensional measure of the unit sphere $S^{n-1}.$ We can also define a more general ring domain.
\begin{definition}
    A \textit{ring domain} is a domain whose complement consists of a bounded and an unbounded component which we denote as $D_0$ and $D_1$, respectively. The boundary components of the ring are then $\partial D_0$ and $\partial D_1.$ We denote the ring domain by $R(D_0,D_1).$ 
\end{definition}

When considering ring domains in particular, there is a function due to V\"{a}is\"{a}l\"{a} \cite{vaisala2006lectures} which will be of particular use for estimates. 

\begin{definition}
    Given $r>0,$ we let $\Phi_n(r)$ be the set of all rings $A = R(D_0,D_1)$ in $S^n,$ where $D_0,D_1\subset S^n$ are the complementary components of the ring, with the following properties:
    \begin{itemize}
        \item[(a)] $D_0$ contains the origin and a point $a$ such that $\abs{a} = 1.$
        \item[(b)] $D_1$ contains $\infty$ and a point $b$ such that $\abs{b} = r.$ 
    \end{itemize}
    We denote
    \begin{align*}
        \mathcal{H}_n(r) & = \inf M(\Gamma_A),
    \end{align*}
    over all rings $A\in \Phi_n(r).$
\end{definition}

\begin{theorem} [\cite{vaisala2006lectures}, Theorem 30.1]
    \label{VHn}
    The function $\mathcal{H}_n:(0,\infty)\rightarrow \mathbb{R}$ has the following properties:
    \begin{itemize}
        \item[(a)] $\mathcal{H}_n$ is decreasing.
        \item[(b)] $\lim_{r\rightarrow \infty} \mathcal{H}_n(r) = 0.$
        \item[(c)] $\lim_{r\rightarrow 0} \mathcal{H}_n(r) = \infty.$
        \item[(d)] $\mathcal{H}_n(r)>0.$
    \end{itemize}
\end{theorem}
We now transition to discussing some properties and results of both quasiregular and quasiconformal maps in the setting of the quasihyperbolic metric. The following result uncovers a local geometric property for quasiregular maps.

\begin{theorem}[\cite{fletcher2016superattracting}, Theorem 1.1]
    \label{FN2014}
    Let $E$ a proper subdomain of $\mathbb{R}^n$ equipped with the Euclidean distance function. Further, suppose $x_0\in E$ and $f:E\rightarrow \mathbb{R}^n$ a non-constant quasiregular map. Then, there exist $C_0>1$ and $r_0>0$ such that, for all $0<T_0\leq 1$ and all $r\in (0,r_0),$
    \begin{align*}
        T_0^{-\mu} & \leq \frac{L_f(x_0,r)}{\ell_f(x_0,T_0r)} \leq \frac{C_0^2}{T_0^{-\nu}},
    \end{align*}
    where $\mu = (i(x_0,f)/K_I(f)^{1/(n-1)})$ and $\nu = (K_O(f)i(x_0,f))^{1/(n-1)}.$ Moreover, $C_0$ depends only on $n,K_O(f)$ and $i(x_0,f).$
\end{theorem}
A corollary to Rickman's well-known $K_O$-inequality for quasiregular mappings \cite[Thm II.2.4]{rickman2012quasiregular}  provides a stronger version of the $K_O$-inequality for quasiconformal mappings.
\begin{proposition}[\cite{rickman2012quasiregular}, Corollary II.2.7]
    \label{KO-inequality}
    If $f:D\rightarrow D'$ is quasiconformal and $\Gamma$ is a path family in $D,$ then 
    \begin{align*}
        M(\Gamma) & \leq K_O(f)M(f(\Gamma)),
    \end{align*}
    where $f(\Gamma)$ denotes the image of $\Gamma$ under $f.$
\end{proposition}

Focusing on quasiconformal maps allows us more insight into the behavior of the map locally and globally. In the following result, Gehring and Osgood \cite{gehring1976quasiconformally} reveal that quasiconformal maps on domains equipped with the quasihyperbolic metric are globally Lipschitz but locally H\"{o}lder.

\begin{theorem} 
    \label{GO}
    Let $X,Y\subset \mathbb{R}^n$ be domains and $f:X\rightarrow Y$ an $M$-quasiconformal map. Then, there exist constants $C_1,C_2>0$ depending only on $n$ and $M$ such that
    \begin{align*}
        k_{Y}(f(x_1),f(x_2)) & \leq C_1\max\{k_X(x_1,x_2), k_X(x_1,x_2)^{\frac{1}{\alpha}}\}
    \end{align*}
    and
    \begin{align*}
        k_{Y}(f(x_1),f(x_2)) & \geq C_2\min\{k_X(x_1,x_2), k_X(x_1,x_2)^\alpha\}
    \end{align*}
    for all $x_1,x_2\in X,$ where $\alpha = M^{\frac{1}{n-1}}.$
\end{theorem}

\subsection{Quasisymmetry and Weak Quasisymmetry} 
Recall that a homeomorphism between metric spaces $f:(X,d_X)\rightarrow (Y,d_Y)$ is $L$-bi-Lipschitz if there exists $L\geq 1$ such that for all $x,y\in X$ we have $\frac{1}{L}d_X(x,y)\leq d_Y(f(x),f(y)) \leq Ld_X(x,y).$ Bi-Lipschitz maps can only distort distances by a bounded amount. To consider a more general class of mappings, we want to allow bounded distortion of relative distances instead.

\begin{definition}
    Let $f:(X,d_X)\rightarrow (Y,d_Y)$ be a homeomorphism. Then $f$ is said to be $\eta$\textit{-quasisymmetric} if there exists an increasing homeomorphism $\eta:[0,\infty)\rightarrow [0,\infty)$ such that for every distinct triple of points $x,y,z\in X,$ we have
    \begin{align*}
        \frac{d_Y(f(x),f(y))}{d_Y(f(x),f(z))} & \leq \eta\left(\frac{d_X(x,y)}{d_X(x,z)}\right).
    \end{align*}
    
\end{definition}

There is a weaker version of the above definition. Proving that the weaker version holds for a function can potentially act as a stepping stone to showing it is $\eta$-quasisymmetric. This is because in certain settings \cite[Thm. 10.19]{heinonen2001lectures}, the two definitions do turn out to be equivalent. 

\begin{definition}
    Let $f:(X,d_X)\rightarrow (Y,d_Y)$ be a map. We say $f$ is \textit{weakly $H$-quasisymmetric} if there exists $H\geq 1$ such that for every distinct triple $x,y,z\in X,$ 
    \begin{align*}
        d_Y(f(x),f(y)) & \leq Hd_Y(f(x),f(z)),
    \end{align*}
    whenever $d_X(x,y)\leq d_X(x,z).$
\end{definition}

However, this version of the weaker definition is not always ideal to work with. The following proposition provides an equivalent definition to weak quasisymmetry.

\begin{proposition}
    \label{weakqsequiv}
    Let $f:(X,d_X)\rightarrow (Y,d_Y)$ be a map. Then $f$ is weakly $H$-quasisymmetric if and only if for every $x\in X$ and $r>0$ we have $f(B_X(x,r))\subset B_Y(f(x),Hl_f(x,r)).$
\end{proposition}
\begin{proof}
    To begin, let us assume $f$ is weakly $H$-quasisymmetric. Let $x\in X, r>0,$ and $v\in f(B_X(x,r)).$ Then, there exists $y\in X$ such that $f(y) = v.$ Choose $z\in X$ such that $d(x,z) = r$ and $d_Y(f(x),f(z)) = \ell_f(x,r).$ If $d_X(x,y)\leq r = d_X(x,z),$ weak $H$-quasisymmetry of $f$ implies
    \begin{align*}
        d_Y(f(x),f(y)) &\leq Hd_Y(f(x),f(z)) = H\ell_f(x,r).
    \end{align*}
    Hence, $v\in B_Y(f(x),H\ell_f(x,r)).$\\
    Conversely, let $x,y,z\in X$ be distinct points such that $d_X(x,y)\leq d_X(x,z).$ Set $r = d_X(x,z).$ Then, $x,y\in \overline{B_X(x,r)}.$  By assumption, $f(x),f(y)\in f(\overline{B_X(x,r)})\subset \overline{B_Y(f(x),H\ell_f(x,r))},$ so
    \begin{align*}
        d_Y(f(x),f(y))& \leq H\ell_f(x,r).
    \end{align*}
     As $d_Y(f(x),f(z))\geq \ell_f(x,r),$ we have
    \begin{align*}
        d_Y(f(x),f(y))& \leq H\ell_f(x,r)\leq Hd_Y(f(x),f(z)),
    \end{align*}
    so $f$ is weakly $H$-quasisymmetric.
\end{proof}

\subsection{Families of Continuous Functions}
Let $X,Y$ be subsets of $S^n$ and $d_X,d_Y$ conformal distance functions. We define $C(X,Y)$ as the set of all continuous functions from $X$ to $Y.$ Now, we can make $C(X,Y)$ into a metric space with the distance function defined the following way: let $(K_i)_{i=0}^\infty$ be a compact exhaustion of $X$ and for $f,g\in C(X,Y)$ set,
\begin{align*}
    d_i(f,g) & = \sup\{d_Y(f(x),g(x)):x\in K_i\}
\end{align*}
and
\begin{align*}
    d_{X,Y}(f,g) & = \sum_{i=1}^\infty \frac{1}{2^i}\left(\frac{d_i(f,g)}{1+d_i(f,g)}\right).
\end{align*}
Then, $d_X(f_m,f)\rightarrow 0$ if and only if $f_m\rightarrow f$ uniformly on compact subsets of $X.$
The topology which arises from this distance function is called the topology of uniform convergence on compact sets.

It will be of importance to discuss the relative compactness of families of functions. We say that a family of functions $\mathcal{F}\subset C(X,Y)$ is relatively compact in $C(X,Y)$ if its closure is compact. Just above we saw that $C(X,Y)$ can be viewed as a metric space, and sequential compactness coincides with compactness in metric spaces. Therefore, $\mathcal{F}$ is relatively compact if and only any sequence $f_m\in \mathcal{F}$ has a subsequence which converges uniformly on compact subsets to an element of $C(X,Y).$ This concept will be key when combined with ideas from the following definitions and results.

\begin{definition}
    Let $\omega:[0,\infty)\rightarrow [0,\infty)$ be a continuous increasing function with $\omega(0) = 0.$ We call a map $f:X\rightarrow Y$ $\omega$-continuous if $f$ has \textit{modulus of continuity} $\omega,$ that is, if
    \begin{align*}
        d_Y(f(x),f(y)) & \leq \omega(d_X(x,y))
    \end{align*}
    for all $x,y\in X.$
\end{definition}

\begin{definition}
    A family $\mathcal{F}\subset C(X,Y)$ is called:
    \begin{itemize}
        \item[(i)]uniformly $\omega$-continuous if there exists $L>0$ such that for all $x,y\in X$ and all $f\in \mathcal{F},$
        \begin{align*}
            d_Y(f(x),f(y)) & \leq L\omega(d_X(x,y));
        \end{align*}

        \item[(ii)]uniformly $\omega$-continuous on compact sets if for each compact set $E\subset X,$ there exists $L>0$ such that for all $x,y\in E$ and all $f\in \mathcal{F},$
        \begin{align*}
            d_Y(f(x),f(y))&\leq L(d_X(x,y));
        \end{align*}

        \item[(iii)] locally uniformly $\omega$-continuous if for each $x_0\in X,$ there exists $r,L>0$ such that for all $x,y\in B_X(x_0,r)$ and all $f\in \mathcal{F},$
        \begin{align*}
            d_Y(f(x),f(y)) & \leq L\omega(d_X(x,y)).
        \end{align*}
    \end{itemize}
    \end{definition}
    
    The following proposition tells us that parts (ii) and (iii) of the above definition are equivalent.
    \begin{proposition}[\cite{fletcher2024normal}, Proposition 2.6]
        \label{FNprop2.6}
        Let $X,Y$ subdomains of $S^n$ equipped distance functions $d_X,d_Y,$ respectively both arising from conformal metrics. A family $\mathcal{F}\subset C(X,Y)$ is locally uniformly $\omega$-continuous if and only if it is uniformly $\omega$-continuous on compact sets. 
    \end{proposition}

    Recall the following definition.
    \begin{definition}
        Let $X,Y$ subdomains of $S^n$ with respective distance functions $d_X,d_Y$ each arising from a conformal metric. Let $\mathcal{F}\subset C(X,Y).$
        \begin{itemize}
            \item[(i)] $\mathcal{F}$ is \textit{equicontinuous at} $x_0\in X$ if for every $\varepsilon >0$ there exists $\delta>0$ such that\\ $d_Y(f(x),f(x_0))<\varepsilon$ whenever $d_X(x,x_0)<\delta$ and $f\in \mathcal{F}.$

            \item[(ii)]$\mathcal{F}$ is \textit{equicontinuous on} $X$ if it is equicontinuous at each point $x_0\in X.$ 
        \end{itemize}
    \end{definition}


    We now consider a version of the Arzela-Ascoli theorem for our setting (\cite[Theorem 47.1]{munkrestopology}).
    
    \begin{theorem}
        \label{FNTheorem2.8}
        Let $\mathcal{F}$ be a family of continuous functions from a locally compact Hausdorff metric space $X$ to a metric space $Y.$ Then $\mathcal{F}$ is relatively compact in $C(X,Y)$ if and only if
        \begin{itemize}
            \item[(i)]the family $\mathcal{F}$ is equicontinuous on $X$ and

            \item[(ii)]for every $x\in X,$ the orbit $\mathcal{F}(x) = \{f(x):f\in \mathcal{F}\}$ is relatively compact in $Y.$ 
        \end{itemize}
    \end{theorem}

    Viewing $C(X,Y)$ as a metric space allows us to reinterpret the Arzela-Ascoli theorem for families $\mathcal{F}\subset C(X,Y)$ which are locally uniformly $\omega$-continuous. This modified version of Arzela-Ascoli provides us with geometric insight that will be of use in Section $4.$
 
    \begin{theorem}[\cite{fletcher2024normal},Theorem 2.9]
        \label{FNTheorem2.9}
        Let $X,Y$ be subdomains of $S^n$ with conformal metrics and associated distance functions $d_X,d_Y$ respectively and suppose the metric on $Y$ is complete. Let $\mathcal{F}\subset C(X,Y)$ be locally uniformly $\omega$-continuous. Then $\mathcal{F}$ is relatively compact in $C(X,Y)$ if and only if there exists $x_0\in X$ such that $\mathcal{F}(x_0) = \{f(x_0):f\in \mathcal{F}\}$ is relatively compact in $Y.$
    \end{theorem}

    \subsection{Normal Families}As stated in our introduction, Beardon and Minda \cite{beardon2014normal} characterize normal families in terms of a locally uniform Lipschitz condition with a view to applications to families of holomorphic functions. The first author and Nicks \cite{fletcher2024normal} then generalized the Beardon and Minda viewpoint to the higher dimensional setting, and a local uniform H\"{o}lder condition with a view to applications to families of quasiregular maps. Here we take a moment to highlight some of this work, as we will be relying on it in Section $4.$  

    \begin{definition}
        If $X,Y$ are subdomains of $S^n,$ then for $M\geq 1,$ denote by $\mathcal{Q}_M(X,Y)$ the subset of $C(X,Y)$ consisting of all $M$-quasiregular mappings from $X$ to $Y.$
    \end{definition}
    Now, recall the following definition of a family being normal relative to a domain from \cite{beardon2014normal}.

    \begin{definition}
        \label{normrelto}
        Suppose $X,Y,Z$ are domains in $S^n$ each equipped with conformal metrics, where $Y\subset Z$ and $\mathcal{F}\subset C(X,Y).$ Then $\mathcal{F}$ is a \textit{normal family relative to $Z$} if $\mathcal{F}$ is relatively compact in $C(X,Z)$ and the closure of $\mathcal{F}$ in $C(X,Z)$ is the closure of $\mathcal{F}$ in $C(X,Y)$ together possibly with constant maps into $\partial Y,$ viewing the boundary of $Y$ as a subset of $Z.$
    \end{definition}

    There are then several equivalent definitions of normal families.
    
    \begin{definition}
        Let $X,Y$ be subdomains of $S^n$ with conformal metrics, let $M\geq 1$ and $\mathcal{F}\subset \mathcal{Q}_M(X,Y).$ Then we say $\mathcal{F}$ is a normal family if any (and hence all) of the following equivalent statements hold:
        \begin{itemize}
            \item[(i)]$\mathcal{F}$ is a normal family relative to $S^n$ in the sense of Definition \ref{normrelto};

            \item[(ii)]$\mathcal{F}$ is relatively compact in $C(X,S^n);$

            \item[(iii)]every sequence $(f_m)$ in $\mathcal{F}$ has subsequence that converges uniformly on compact subsets of $X,$ in the spherical metric, to a limit function $f:X\rightarrow S^n.$ 
        \end{itemize}
    \end{definition}
    
    The following is one of the main results from \cite{fletcher2024normal}. It is incredibly useful in the case where $\mathcal{F}\subset \mathcal{Q}_M(X,S^n),$ as it relates normality of a family to local uniform $\omega$-continuity.
    
    \begin{theorem}[\cite{fletcher2024normal}, Theorem 3.7]
        \label{FNTheorem3.7}
        Let $X\subset S^n$ be a domain equipped with distance function $d_X$ arising from a conformal metric and let $\mathcal{F}\subset \mathcal{Q}_M(X,Y)$ be a family of $M$-quasiregular mappings defined on $X$ with image contained in $Y\subset S^n$ equipped with distance function $d_Y$ arising from a complete conformal metric. Then $\mathcal{F}$ is relatively compact in $C(X,Y)$ if and only if
        \begin{itemize}
            \item[(i)]$\mathcal{F}$ is locally uniformly $\omega$-continuous, with $\omega(t) = t^\alpha, \alpha = M^{1/(1-n)}$ and

            \item[(ii)]there exists $x_0\in X$ such that $\mathcal{F}(x_0) = \{f(x_0):f\in \mathcal{F}\}$ is relatively compact in $Y.$  
        \end{itemize}
        In particular, if $Y$ is $S^n$ with the spherical metric and if $\mathcal{F}\subset \mathcal{Q}_M(X,S^n)$ is a family of $M$-quasiregular mappings, then $\mathcal{F}$ is normal if and only if $\mathcal{F}$ is locally uniformly $\omega$-continuous, with $\omega(t) = t^\alpha, \alpha = M^{1/(1-n)}.$
    \end{theorem}

    If we consider a family $\mathcal{F}\subset \mathcal{Q}_M(X,\mathbb{R}^n)$ instead then we can say $\mathcal{F}$ is normal if and only if every sequence in $\mathcal{F}$ has a subsequence which converges uniformly on compact sets to either an element of $\mathcal{Q}_M(X,\mathbb{R}^n)$ or diverges to infinity. From this realization, we form a new definition. 

    \begin{definition}
        Let $\mathcal{F}\subset \mathcal{Q}_M(X,\mathbb{R}^n),$ where $X$ is domain with a distance function $d_X$ arising from a conformal metric and $\mathbb{R}^n$ is equipped with the Euclidean metric. We say $\mathcal{F}$ is \textit{finitely normal} if every sequence in $\mathcal{F}$ has a subsequence which converges uniformly on compact sets  to an element of $\mathcal{Q}_M(X,\mathbb{R}^n).$ 
    \end{definition}

    Then, there exists a result similar to that of Theorem \ref{FNTheorem3.7} but for the finitely normal setting.

    \begin{corollary}[\cite{fletcher2024normal},Corollary 3.10]
        \label{FNCorollary3.10}
        Let $X\subset S^n$ be a domain equipped with distance function $d_X$ arising from a conformal metric and let $\mathcal{F}\subset \mathcal{Q}_M(X,\mathbb{R}^n)$ be a family of $M$-quasiregular mappings. Then $\mathcal{F}$ is finitely normal if and only if $\mathcal{F}$ is locally uniformly $\omega$-continuous, with $\omega(t) = t^\alpha, \alpha = M^{1/(1-n)},$ and $\{f(x_0):f\in \mathcal{F}\}$ is bounded for some $x_0\in X.$ 
    \end{corollary}

\section{Quasiconformal Implies Quasisymmetric on Uniformly K-QCH Domains}

In \cite{ackermann2021quasiconformality}, Ackermann and the first author prove a map  $f:\mathbb{B}^n\rightarrow \mathbb{B}^n$ is quasiconformally if and only if it is $\eta$-quasisymmetric, where we equip $\mathbb{B}^n$ with the hyperbolic metric. They then pose the question of whether the result generalizes to the setting of an arbitrary domain $X\subset \mathbb{R}^n$ equipped with the quasihyperbolic metric. In this section, we investigate that question. For Ackermann and the first author the bulk of the work was in proving the forward implication. To do so, they rely heavily on the fact that the set of M\"{o}bius maps are a transitive collection of conformal hyperbolic isometries on $\mathbb{B}^n.$ So, to generalize this strategy to a domain $X\subset \mathbb{R}^n$ equipped with the quasihyperbolic metric, it is natural to require $X$ to be a uniformly QCH domain. 

In our situation, to prove quasiconformality implies quasisymmetry, we must first prove a quasiconformal map is weak quasisymmetric and then prove that quasisymmetry and weak quasisymmetry are equivalent. Proving a quasiconformal map is weak quasisymmetric requires a local result which shows that quasiconformal maps satisfy the $\eta$-quasisymmetric condition when we fix a base point.

\begin{lemma}
    \label{localqs}
    Let $X,Y\subset \mathbb{R}^n$ domains. Suppose $f:(X,k_X)\rightarrow (Y,k_Y)$ is a $K$-quasiconformal map, $T>0,$ and $x_0\in X.$ Then there exists a constant $\xi$ depending on $x_0,T,n,$ and $K$ such that
    \begin{align*}
        \frac{L_f(x_0,r)}{\ell_f(x_0,Tr)} & \leq \xi(x_0,n,K,T),
    \end{align*}
    for all $r>0.$
\end{lemma}
\begin{proof}
    Let $x_0\in X.$ Given $r,T>0$ find $y,z\in X$ such that $k_X(x_0,y) = r,$ $L_f(x_0,r) = k_Y(f(x_0),f(y)),$ $k_X(x_0,z) = Tr$ and $\ell_f(x_0,Tr) = k_Y(f(x_0),f(z)).$ As the Euclidean and quasihyperbolic distance functions both arise from conformal metrics, they are locally equivalent. So, we can use Theorem \ref{FN2014} in the quasihyperbolic setting to say there exist $C_0$ and $r_0>0$ such that
    \begin{align}
        \frac{k_Y(f(x_0),f(y))}{k_Y(f(x_0),f(z))} & = \frac{L_f(x_0,r)}{\ell_f(x_0,Tr)} \leq \frac{C_0^2}{T^\nu},\label{Final I}
    \end{align}
    for all $0<T\leq 1$ and $r\in (0,r_0),$ where $\nu = (K_0(f)i(x_0,f))^{1/(n-1)}.$ In other words, we have the upper bound above if $y,z\in \overline{B_X(x_0,r_0)}.$ We now need to consider when $y$ or $z$ are in $X\setminus\overline{B_X(x_0,r_0)}.$ To do so, we need two cases depending on whether $r_0<1$ or $r_0\geq 1.$ These two cases each have two sub-cases considering whether $T<1$ or $T\geq 1.$\\
    \textbf{Case 1:} Suppose $r_0<1$ and $T<1.$ Then $Tr<r.$ We require $y$ or $z$ to be in $ X\setminus\overline{B_X(x_0,r_0)},$ so we can deduce that $k_X(x_0,y) = r>r_0.$ Recall that $\alpha = K^{1/(n-1)}.$ Since $r_0<r$ and $\frac{1}{\alpha} - 1<0,$ notice
    \begin{align}
        r^{\frac{1}{\alpha}} & = r^{\frac{1}{\alpha}-1}\cdot r \leq r_0^{\frac{1}{\alpha}-1} \cdot r.\label{C11}
    \end{align}
    Moreover, $r_0^{\frac{1}{\alpha}-1}>1$ which implies $r_0^{\frac{1}{\alpha}-1} \cdot r\geq r.$ Observe,
    \begin{align}
        (Tr)^\alpha & = (Tr)^{\alpha-1}\cdot Tr \geq (Tr_0)^{\alpha-1}\cdot Tr.\label{C12}
    \end{align}
    Since $T,r_0<1,$ it follows that
    \begin{align*}
        (Tr_0)^{\alpha-1}\cdot Tr & < Tr.
    \end{align*}
    Using Theorem \ref{GO}, (\ref{C11}), and (\ref{C12})
    \begin{align}
        \frac{k_Y(f(x_0),f(y))}{k_Y(f(x_0),f(z))}& \leq \frac{C_1\max\{k_X(x_0,y), k_X(x_0,y)^{\frac{1}{\alpha}}\}}{C_2\min\{k_X(x_0,y), k_X(x_0,y)^{\alpha}\}}\nonumber\\
                        & = \frac{C_1\max\{r,r^\frac{1}{\alpha}\}}{C_2\min\{Tr,(Tr)^\alpha\}}\nonumber\\
                        & \leq \frac{C_1\max\{r,r_0^{\frac{1}{\alpha}-1}r\}}{C_2\min\{Tr,(Tr_0)^{\alpha-1}Tr\}}\nonumber\\
                        & = \frac{C_1 r_0^{\frac{1}{\alpha}-1}r}{C_2(Tr_0)^{\alpha-1}Tr}\nonumber\\
                        & = \frac{C_1 r_0^{\frac{1}{\alpha}-\alpha}}{C_2 T^\alpha}.\label{Final II}
    \end{align}
    \textbf{Case 2:} Consider when $r_0<1$ and $T\geq 1.$ Then $Tr>r.$ We require $y$ or $z$ to not be in $\overline{B_X(x_0,r_0)},$ so $k_X(x_0,z) = Tr>r_0.$ Notice,
    \begin{align}
        (Tr)^\alpha & = (Tr)^{\alpha-1}\cdot Tr \geq r_0^{\alpha-1}\cdot Tr.\label{C21}
    \end{align}
    Since $r>\frac{r_0}{T}$ and $\frac{1}{\alpha}-1<0,$ observe that
    \begin{align}
        r^{\frac{1}{\alpha}} & = r^{\frac{1}{\alpha}-1}r\leq \left(\frac{r_0}{T}\right)^{\frac{1}{\alpha}-1}\cdot r. \label{C22}
    \end{align}
    Finally, note that $r_0^{\alpha-1}<1$ and $\left(\frac{r_0}{T}\right)^{\frac{1}{\alpha}-1} >1$ since $\frac{r_0}{T}<r_0<1.$ Then, by Theorem \ref{GO}, (\ref{C21}) and (\ref{C22}), we have
    \begin{align}
        \frac{k_Y(f(x_0),f(y))}{k_Y(f(x_0),f(z))}& \leq \frac{C_1\max\{k_X(x_0,y), k_X(x_0,y)^{\frac{1}{\alpha}}\}}{C_2\min\{k_X(x_0,y), k_X(x_0,y)^{\alpha}\}}\nonumber\\
                        & = \frac{C_1\max\{r,r^\frac{1}{\alpha}\}}{C_2\min\{Tr,(Tr)^\alpha\}}\nonumber\\
                        & \leq \frac{C_1\max\{r,\left(\frac{r_0}{T}\right)^{\frac{1}{\alpha}-1}r\}}{C_2\min\{Tr,r_0^{\alpha-1}Tr\}}\nonumber\\
                        & \leq \frac{C_1 \left(\frac{r_0}{T}\right)^{\frac{1}{\alpha}-1}r}{C_2r_0^{\alpha-1}Tr}\nonumber\\
                        & = \frac{C_1r_0^{\frac{1}{\alpha}-\alpha}}{C_2T^\frac{1}{\alpha}}.\label{Final III}
    \end{align}
    \textbf{Case 3:} Suppose that $r_0\geq 1$ and $T<1.$ Then, $Tr<r$ so we can deduce that $k_X(x_0,y) = r>r_0\geq 1.$ Then, 
    \begin{align}
        r^{\frac{1}{\alpha}} & = r^{\frac{1}{\alpha}-1}\cdot r \leq r_0^{\frac{1}{\alpha}-1}\cdot r.\label{C31}
    \end{align}
    Assuming $r_0\geq 1$ implies $r_0^{\frac{1}{\alpha}-1}<1$ so $r_0^{\frac{1}{\alpha}-1}\cdot r < r.$ Also, since $r>r_0,$
    \begin{align}
        (Tr)^\alpha & = (Tr)^{\alpha-1}\cdot Tr \geq (Tr_0)^{\alpha-1}\cdot Tr.\label{C32}
    \end{align}
    Using Theorem \ref{GO}, (\ref{C31}), and (\ref{C32}) we have
    \begin{align*}
         \frac{k_Y(f(x_0),f(y))}{k_Y(f(x_0),f(z))}& \leq \frac{C_1\max\{k_X(x_0,y), k_X(x_0,y)^{\frac{1}{\alpha}}\}}{C_2\min\{k_X(x_0,y), k_X(x_0,y)^{\alpha}\}}\\
                        & = \frac{C_1\max\{r,r^\frac{1}{\alpha}\}}{C_2\min\{Tr,(Tr)^\alpha\}}\\
                        & \leq \frac{C_1\max\{r,(r_0)^{\frac{1}{\alpha}-1}r\}}{C_2\min\{Tr,(Tr_0)^{\alpha-1}Tr\}}\\
                        & = \frac{C_1r}{C_2\min\{Tr, (Tr_0)^{\alpha-1}Tr\}}.
    \end{align*}
    If $Tr_0<1$ then 
    \begin{align}
        \frac{C_1r}{C_2\min\{Tr, (Tr_0)^{\alpha-1}Tr\}} & = \frac{C_1r}{C_2(Tr_0)^{\alpha-1}Tr} = \frac{C_1}{C_2 T^\alpha r_0^{\alpha-1}}.\label{Final IV}
    \end{align}
    Alternatively, if $Tr_0\geq1$ then
    \begin{align}
        \frac{C_1r}{C_2\min\{Tr, (Tr_0)^{\alpha-1}Tr\}} & = \frac{C_1r}{C_2Tr} = \frac{C_1}{C_2 T}.\label{Final V}
    \end{align}
    \textbf{Case 4}: Suppose $r_0\geq 1$ and $T\geq 1.$ Then, $Tr>r$ so $k_X(x_0,z) = Tr>r_0.$ Notice,
    \begin{align}
        (Tr)^\alpha & = (Tr)^{\alpha-1}Tr\geq (r_0)^{\alpha-1}Tr.\label{C41}
    \end{align}
    We assumed $r>\frac{r_0}{T},$ so
    \begin{align}
        r^\frac{1}{\alpha} & = r^{\frac{1}{\alpha}-1}\cdot r \leq \left(\frac{r_0}{T}\right)^{\frac{1}{\alpha}-1}\cdot r.\label{C42}
    \end{align}
    Thus, by Theorem \ref{GO}, (\ref{C41}), and (\ref{C42})
    \begin{align*}
        \frac{k_Y(f(x_0),f(y))}{k_Y(f(x_0),f(z))}& \leq \frac{C_1\max\{k_X(x_0,y), k_X(x_0,y)^{\frac{1}{\alpha}}\}}{C_2\min\{k_X(x_0,y), k_X(x_0,y)^{\alpha}\}}\\
                        & = \frac{C_1\max\{r,r^\frac{1}{\alpha}\}}{C_2\min\{Tr,(Tr)^\alpha\}}\\
                        & \leq \frac{C_1\max\{r,\left(\frac{r_0}{T}\right)^{\frac{1}{\alpha}-1}r\}}{C_2\min\{Tr,(r_0)^{\alpha-1}Tr\}}\\
                        & = \frac{C_1\max\{r,\left(\frac{r_0}{T}\right)^{\frac{1}{\alpha}-1}r\}}{C_2 Tr},
    \end{align*}
    where the last line follows because $r_0^{\alpha-1}Tr\geq Tr$ since $Tr>r_0\geq 1$ in this case. Now, if $r_0<T,$ then $\left(\frac{r_0}{T}\right)^{\frac{1}{\alpha}-1}\cdot r > r,$ so
    \begin{align}
        \frac{C_1\max\{r,\left(\frac{r_0}{T}\right)^{\frac{1}{\alpha}-1}r\}}{C_2Tr} & = \frac{C_1\left(\frac{r_0}{T}\right)^{\frac{1}{\alpha}-1}r}{C_2Tr} = \frac{C_1r_0^{\frac{1}{\alpha}-1}}{C_2T^\frac{1}{\alpha}}.\label{Final VI}
    \end{align}
    If instead $r_0\geq T,$ then $\left(\frac{r_0}{T}\right)^{\frac{1}{\alpha}-1}\cdot r \leq r.$ Thus,
    \begin{align}
        \frac{C_1\max\{r,\left(\frac{r_0}{T}\right)^\frac{1}{\alpha}r\}}{C_2Tr}
                        & = \frac{C_1r}{C_2Tr} = \frac{C_1}{TC_2}.\label{Final VII}
    \end{align}
    Hence, for all $y,z\in X$ such that $k_X(x_0,y) = r$ and $k_X(x_0,z) = Tr,$ by $(\ref{Final I}),(\ref{Final II}),(\ref{Final III}), (\ref{Final IV}),(\ref{Final V}),(\ref{Final VI})$ and $(\ref{Final VII}),$
    \begin{align*}
        \frac{L_f(x_0,r)}{\ell_f(x_0,Tr)} & \leq \max\left\{\frac{C_0^2}{T^\nu}, \frac{C_1 r_0^{\frac{1}{\alpha}-\alpha}}{C_2 T^\alpha}, \frac{C_1r_0^{\frac{1}{\alpha}-\alpha}}{C_2T^\frac{1}{\alpha}}, \frac{C_1}{C_2 T}, \frac{C_1r_0^{\frac{1}{\alpha}-1}}{C_2T^\frac{1}{\alpha}}, \frac{C_1}{C_2 T^\alpha r_0^{\alpha-1}}\right\} = \xi(x_0,n,K,T),
    \end{align*}
    and clearly $\xi$ only depends on $K,n,x_0,$ and $T.$ 
\end{proof}

Next, we use the preceding lemma to prove quasiconformal maps are weak quasisymmetric.

\begin{theorem}
    \label{QCHweakqs}
    Let $X,Y$ proper subdomains of $\mathbb{R}^n$ and suppose $X$ is a uniformly $M$-QCH domain. If $f:(X,k_X)\rightarrow (Y,k_Y)$ is $K$-quasiconformal then it is weakly $H$-quasisymmetric.
\end{theorem}

\begin{proof}
    Let $X$ be a uniformly $M$-QCH domain and $f$ an $K$-quasiconformal map. By Proposition \ref{weakqsequiv}, it is sufficient to show $f(B_X(x,r))\subset B_Y(f(x),Hl_f(x,r))$ for all $x\in X$ and $r>0.$ Showing such inclusion is equivalent to proving
    \begin{align*}
        \frac{L_f(x,r)}{l_f(x,r)} & \leq H,
    \end{align*}
    for all $x\in X$ and $r>0.$ To do so, we consider two cases: $r\geq 1$ and $r<1.$\\
    \textbf{Case 1:} Suppose $r\geq 1.$ There exist $y,z\in X$ such that
    \begin{align*}
        L_f(x,r) & = k_Y(f(x),f(y)),
    \end{align*}
    and
    \begin{align*}
        l_f(x,r) & = k_Y(f(x),f(z)).
    \end{align*}
    By Theorem \ref{GO}, there exist constants $C_1,C_2>0$ depending only on $n$ and $K$ such that
    \begin{align*}
        L_f(x,r) & = k_X(f(x),f(y))\leq C_1k_X(x,y) = C_1r,
    \end{align*}
    and
    \begin{align*}
        l_f(x,r) & = k_Y(f(x),f(z)) \geq C_2k_X(x,z) = C_2r.
    \end{align*}
    Thus, for all $x\in X$ and $r\geq 1$ we have
    \begin{align*}
        \frac{L_f(x,r)}{l_f(x,r)} & \leq \frac{C_1r}{C_2r} = \frac{C_1}{C_2}.
    \end{align*}
    \textbf{Case 2:} Suppose $r<1.$ Recall that $G$ is the set of transitive $M$-quasiconformal mappings arising from $X.$ Fix $x_0\in X$ and find $g\in G$ such that $g(x_0) = x.$ Set $h = f\circ g.$ Note that $h$ is an $MK$-quasiconformal map such that $h(x_0) = f(x).$ Then, it follows that
    \begin{align*}
        L_f(x,r)& \leq L_{h}(x_0,L_{g^{-1}}(x,r)),
    \end{align*}
    and
    \begin{align*}
        l_f(x,r)&\geq l_{h}(x_0,l_{g^{-1}}(x,r)).
    \end{align*}
    Thus,
    \begin{align*}
        \frac{L_f(x,r)}{l_f(x,r)}& \leq \frac{ L_{h}(x_0,L_{g^{-1}}(x,r))}{l_{h}(x_0,l_{g^{-1}}(x,r))}.
    \end{align*}
    By Lemma \ref{localqs}, there exists a constant $\psi$ dependent on $n,M,K,x_0$ and $T = \ell_{g^{-1}}(x,r)/L_{g^{-1}}(x,r)\leq 1$ such that
    \begin{align*}
        \frac{ L_{h}(x_0,L_{g^{-1}}(x,r))}{l_{h}(x_0,l_{g^{-1}}(x,r))} & = \frac{ L_{h}(x_0,L_{g^{-1}}(x,r))}{l_{h}(x_0,TL_{g^{-1}}(x,r))} \leq \psi(x_0,n,M,K,T).
    \end{align*}
    However, since $T$ is dependent on $x,$ then $\psi$ also depends on choice of $x.$ Let us make $\psi$ independent of $x.$ Removing the dependence of $T$ on $x,$ is equivalent to finding an upper bound on 
    \begin{align*}
        \frac{L_{g^{-1}}(x,r)}{l_{g^{-1}}(x,r)}
    \end{align*}
    which depends only on $n$ and $M.$ To accomplish this, we break into two cases.
    
    \noindent\textbf{Case 2a:} Suppose $L_{g^{-1}}(x,r)\neq l_{g^{-1}}(x,r).$ Set $D_1 = g(B_X(x_0,l_{g^{-1}}(x,r))),$ $D_2 = g(B_X(x_0,L_{g^{-1}}(x,r))),$ and $R = D_2\setminus \overline{D_1}.$ Let $\Gamma$ be the path family joining the boundary components of $R$ while staying in $R.$ Using \cite[Ex. 7.5]{vaisala2006lectures}, we know
    \begin{align*}
        M(g^{-1}(\Gamma)) & \leq \omega_{n-1}\left(\log\left(\frac{L_{g^{-1}}(x,r)}{l_{g^{-1}}(x,r)}\right)\right)^{1-n},
    \end{align*}
    and since $R$ is a ring domain, by Theorem \ref{VHn} 
    \begin{align*}
        0 & <\mathcal{H}_n(1)\leq M(\Gamma).
    \end{align*}
    Now, using Proposition \ref{KO-inequality},
    \begin{align*}
        \mathcal{H}_n(1) & \leq M(\Gamma)\\
                         & \leq K_O(g^{-1})M(g^{-1}(\Gamma))\\
                         & \leq M\cdot\omega_{n-1}\left(\log\left(\frac{L_{g^{-1}}(x,r)}{l_{g^{-1}}(x,r)}\right)\right)^{1-n},
    \end{align*}
    where $K_O(g^{-1})\leq M$ because $g$ is $M$-quasiconformal. Then, notice
    \begin{align*}
        \text{exp}\left(\left(\frac{M\omega_{n-1}}{\mathcal{H}_n(1)}\right)^{\frac{1}{n-1}}\right) & \geq \frac{L_{g^{-1}}(x,r)}{l_{g^{-1}}(x,r)},
    \end{align*}
    which is an upper bound dependent only on $n$ and $M.$  
    
    \noindent\textbf{Case 2b:} Now, we have to consider the case where $L_{g^{-1}}(x,r) = l_{g^{-1}}(x,r)$ separately because the ring $R$ degenerates here and $L_{g^{-1}}(x,r)/l_{g^{-1}}(x,r) = 1.$ Thus,
    \begin{align*}
        \frac{L_{g^{-1}}(x,r)}{l_{g^{-1}}(x,r)} & \leq \max\left\{ \text{exp}\left(\left(\frac{M\omega_{n-1}}{\mathcal{H}_n(1)}\right)^{\frac{1}{n-1}}\right), 1\right\},
    \end{align*}
    for all $x\in X$ and $0<r<1.$ Moreover, this upper bound will hold for any $g\in G$ chosen. Thus, by Lemma \ref{localqs},
    \begin{align*}
     \frac{L_f(x,r)}{l_f(x,r)}& \leq \frac{ L_{h}(x_0,L_{g^{-1}}(x,r))}{l_{h}(x_0,l_{g^{-1}}(x,r))}\\
     & = \frac{ L_{h}(x_0,L_{g^{-1}}(x,r))}{l_{h}(x_0,TL_{g^{-1}}(x,r))}\\
     & \leq \psi(x_0,n,K,M),
    \end{align*}
    where $\psi$ is now dependent only on $x_0,n,M$ and $K.$ We can then set
    \begin{align*}
        H & = \max\left\{\frac{C_1}{C_2}, \psi(x_0,n,K,M)\right\},
    \end{align*}
    and we have
    \begin{align*}
        \frac{L_f(x,r)}{l_f(x,r)} & \leq H,
    \end{align*}
    for all $x\in X$ and $r>0,$ as desired.
\end{proof}

Now that we have shown $f$ is a weak quasisymmetry, we show the notions of weak quasisymmetry and quasisymmetry coincide in our setting. Moreover, we manage to prove such equivalence for the more general setting of geodesic metric spaces than just quasihyperbolic metric spaces. It is well-known that for a homeomorphism $f: (X,d_X)\rightarrow (Y,d_Y),$ weak quasisymmetry and quasisymmetry coincide if $X, Y$ are doubling spaces and $X$ is also connected \cite[Theorem $10.19$]{heinonen2001lectures}. However, such equivalence is not known in general. As the proof strategy for Theorem \ref{qsequivweakqs} is similar to that of \cite[Theorem $10.19$]{heinonen2001lectures}, the result is most likely known to experts in the fields. However, as we were unable to find it in the literature, we provide it in this below for completeness.

\begin{theorem}
    \label{qsequivweakqs}
    Suppose $(X,d_X)$ and $(Y,d_Y)$ are geodesic metric spaces. Then a homeomorphism $f:X\rightarrow Y$ is weakly $H$-quasisymmetric if and only if $f$ is quasisymmetric. 
\end{theorem}

\begin{proof}
    Let $x,a,b\in X$ and first assume that $f:(X,d_X)\rightarrow (Y,d_Y)$ is weakly $H$-quasisymmetric. Set
    \begin{align*}
        t & = \frac{d_X(x,a)}{d_X(x,b)} \hspace{5mm}\mathrm{and}\hspace{5mm} t' = \frac{d_Y(f(x),f(a))}{d_Y(f(x),f(b))}. 
    \end{align*}
    Our goal is to show there exists a homeomorphism $\eta:[0,\infty)\rightarrow [0,\infty)$ such that $t'\leq \eta(t)$ for all $t.$ To do so, we need to consider two cases: $t\geq 1$ and $t<1.$\\ 
    \textbf{Case 1:} Let us first investigate when $t\geq 1.$ To begin, set $r = d_X(x,b)$ and
    \begin{align*}
        \varepsilon & = \frac{rt}{\ceil{t}}\leq r,
    \end{align*}
    where $\ceil{t}$ denotes the ceiling of $t.$ Since $(X,d_X)$ is a geodesic metric space, we can find a geodesic $\gamma$ with endpoints $x$ and $a.$ Moreover, we can find a sequence of $N = \ceil{t}$ points along $\gamma,$ denoted $(a_i)_{i=0}^{N-1},$ with the following three properties:
    \begin{align}
        & a_0 = x,\label{property1a}\\
        &\nonumber\\
        & d_X(a_i,a_{i+1}) = \varepsilon\leq r, \textrm{ for all } 0\leq i\leq N-2,\label{property2a}\\
        &\nonumber\\
        & \sum_{i=1}^{N-2}d_X(a_i,a_{i+1}) = d_X(x,a).\label{property3a}
    \end{align}
    By construction, $d_X(a_i,a_{i+1}) \leq d_X(a_i,a_{i-1})$ for all $0\leq i\leq N-2.$ So, using that $f$ is weakly $H$-quasisymmetric, we have
    \begin{align}
        d_Y(f(a_{i+1}),f(a_i))&\leq Hd_Y(f(a_i),f(a_{i-1})),\label{3.3.1}
    \end{align}
    for all $0\leq i\leq N-1.$ Now, if we repeatedly apply (\ref{3.3.1}) we have
    \begin{align*}
        d_Y(f(a_{i+1}),f(a_i)) & \leq Hd_Y(f(a_i),f(a_{i-1}))\\
                               & \leq H(Hd_Y(f(a_{i-1}),f(a_{i-2})))\\
                               & \vdots\\
                               & \leq H^id_Y(f(a_0),f(a_1))\\
                               & = H^id_Y(f(x),f(a_1)).
    \end{align*}
    We chose $d_X(x,a_1)\leq r = d_X(x,b),$ so again by the weak quasisymmetry of $f,$
    \begin{align}
        d_Y(f(x),f(a_1)) & \leq Hd_Y(f(x),f(b)).\label{3.3.2}
    \end{align}
    Thus, by (\ref{3.3.2}),
    \begin{align}
        d_Y(f(a_i),f(a_{i+1})) & \leq H^id_Y(f(x),f(a_1)) \leq H^{i+1}d_Y(f(x),f(b)).\label{3.3.3.}
    \end{align}
    Finally, using (\ref{property2a}), (\ref{property3a}), and (\ref{3.3.3.}) we have
    \begin{align*}
        d_Y(f(x),f(a)) & \leq \sum_{i=0}^{N-2}d_Y(f(a_i),f(a_{i+1}))\\
                       & \leq \sum_{i=0}^{N-2}H^{i+1}d_Y(f(x),f(b))\\
                       & = (H+H^2+\ldots + H^N)d_Y(f(x),f(b))\\
                       & =(H+H^2+\ldots+H^{\ceil{t}})d_Y(f(x),f(b))\\
                       & \leq (H+H^2+\ldots+H^{t+1})d_Y(f(x),f(b))\\
                       & = (t+1)H^{t+1}d_Y(f(x),f(b)),
    \end{align*}
    where the second to last line follows from $\ceil{t}\leq t+1$ and $H\geq 1.$ Rearranging, we have
    \begin{align*}
        t' & \leq (t+1)H^{t+1}.
    \end{align*}
    which is an increasing function of $t.$\\
    \textbf{Case 2:} Consider when $t<1.$ This implies $d_X(x,a)\leq d_X(x,b).$ By the weak quasisymmetry of $f,$
    \begin{align*}
        t' = \frac{d_Y(f(x),f(a))}{d_Y(f(x),f(b))} & \leq H\leq (t+1)H^{t+1}.
    \end{align*}
    Set $\varphi(t) = (t+1)H^{t+1}.$ Then, $\varphi(t)$ is a function dependent only on $H$ and $t$ such that $t'\leq \varphi(t)$ for all $t.$ Recall that our goal is to find a function of $t$ which also converges to $0$ as $t\rightarrow 0,$ and $\varphi(t)$ does not accomplish this. However, the fact that $t'\leq \varphi(t)$ for all $t$ does allow us to instead assume for this case that $t$ is small to start, say $0<t\leq \frac{1}{3}.$\\
    Since $(Y,d_Y)$ is a geodesic metric space, we can find a geodesic $\gamma'$ with endpoints $f(x)$ and $f(b).$ Moreover, we can choose a sequence of points $(b_i)_{i=0}^s$ on $f^{-1}(\gamma')$ with the following three properties:
    \begin{align}
        & b_0 = b,\label{property1b}\\
        &\nonumber\\
        & d_X(b_i,x) = 3^{-i}d_X(b,x) \textrm{ for all } 0\leq i \leq s-1,\label{property2b}\\
        &\nonumber\\
        & s\geq 2 \textrm{ is the least integer such that } 3^{-s}d_X(b,x)\leq d_X(a,x).\label{property3b}
    \end{align}
    From $(\ref{property3b}),$ we have
    \begin{align*}
        \frac{1}{3^s} & \leq t\\
        3^s & \geq \frac{1}{t}\\
        s & \geq \frac{\log(\frac{1}{t})}{\log(3)}.
    \end{align*}
    Note that $s_0(t) = \frac{\log(\frac{1}{t})}{\log(3)}\rightarrow \infty$ as $t\rightarrow 0.$ Now, let $0\leq i<j<s-1.$ By the triangle inequality, and $(\ref{property2b})$ of our sequence, notice
    \begin{align*}
        d_X(b_i,b_j) & \geq d_X(b_i,x) -d_X(b_j,x)\\
                     & = \frac{1}{3^i}d_X(x,b) - \frac{1}{3^j}d_X(x,b)\\
                     & = \frac{3^{j-i}-1}{3^j}d_X(x,b).
    \end{align*}
    Since $i<j,$ we know $\frac{1}{3^j}<\frac{3^{j-i}-1}{3^j}.$ So, by $(\ref{property2b})$ again,
    \begin{align}
        d_X(b_j,x) & = \frac{1}{3^j}d_X(x,b) \leq \frac{3^{j-i}-1}{3^j}d_X(x,b) \leq d_X(b_i,b_j),\label{3.3.4}
    \end{align}
    for all $0\leq i<j\leq s-1.$ Also, using the triangle inequality, $(\ref{property3b})$ and $(\ref{property2b})$ we have 
    \begin{align*}
        d_X(b_j,a)&\leq d_X(b_j,x)+d_X(x,a)\\
                & \leq d_X(b_j,x) + \frac{1}{3^j}d_X(b,x)\\
                & = \frac{1}{3^j}d_X(x,b) + \frac{1}{3^j}d_X(b,x)\\
                & = \frac{2}{3^j}d_X(b,x).
    \end{align*}    
    Since $i<j,$ then $j-i>1,$ so $\frac{2}{3^j}\leq \frac{3^{j-i}-1}{3^j}.$ Thus, by (\ref{3.3.4})
    \begin{align}
    d_X(b_j,a) & \leq \frac{2}{3^j}d_X(x,b) \leq \frac{3^{j-i}-1}{3^j}d_X(x,b) \leq d_X(b_i,b_j),\label{3.3.5}
    \end{align}
   for all $0\leq i<j<s-1.$ Using the triangle inequality, (\ref{3.3.4}), (\ref{3.3.5}), and that $f$ is weakly $H$-quasisymmetric, notice
   \begin{align*}
       d_Y(f(x),f(a)) & \leq d_Y(f(x),f(b_j)) + d_Y(f(b_j),f(a))\\
                      & \leq Hd_Y(f(b_i),f(b_j))+Hd_Y(f(b_i),f(b_j))\\
                      & = 2Hd_Y(f(b_i),f(b_j)),
   \end{align*}
   for all $0\leq i<j\leq s-1.$ As $d_X(b_i,x)\leq d(x,b)$ for all $i\leq s-1,$ the image of the geodesic containing $(b_i)_{i=0}^{s-1}$ is contained in $\overline{B(f(x),Hd_Y(f(x),f(b)))}.$ Recall Remark \ref{quasihypisquasconvex}, which states that $(Y,d_Y)$ being a geodesic metric space implies it is $1$-quasiconvex. Thus, 
   \begin{align*}
       \sum_{i=0}^{s-2}d_Y(f(x),f(a)) & \leq \sum_{i=0}^{s-2}2Hd_Y(f(b_i),f(b_{i+1}))
   \end{align*}
   and so
   \begin{align*}
       (s-1)d_Y(f(x),f(a)) &\leq 2H\ell(\gamma') = 2Hd_Y(f(x),f(b)).
   \end{align*}
   Rearranging, we have
   \begin{align*}
       (s-1)d_Y(f(x),f(a)) & \leq 2Hd_Y(f(x),f(b))\\
       s-1 & \leq \frac{2H}{t'}\\
       s & \leq \frac{2H}{t'}+1.
   \end{align*}
   Recall that we found $s_0(t) \leq s,$ so
   \begin{align*}
       s_0(t) \leq \frac{2H}{t'}+1,
   \end{align*}
   which implies that
   \begin{align*}
       \frac{\log\left(\frac{1}{t}\right)}{\log(3)} & \leq \frac{2H}{t'}+1\\
       t' & \leq \frac{2H}{\frac{\log\left(\frac{1}{t}\right)}{\log(3)}-1}. 
   \end{align*}
   Hence, we can set $\eta(t) = \frac{2H}{s_0(t)-1},$ and $\eta(t)\rightarrow 0$ as $t\rightarrow 0,$ as desired.\\
   Conversely, if $f$ is quasisymmetric, with homeomorphism $\eta(t),$ then we can choose $H = \eta(1)$ to see that $f$ is weakly $H$-quasisymmetric.
\end{proof}

The previous theorems completed all of the heavy lifting. We can simply string them together to prove the main result of this paper.

\begin{proof}[Proof of Theorem \ref{qcisqs}]
    By Theorem \ref{QCHweakqs}, $f$ is weakly $H$-quasisymmetric. Thus, $f$ is $\eta$-quasisymmetric by Theorem \ref{qsequivweakqs}. Conversely, suppose $f$ is $\eta$-quasisymmetric. Let $x_0\in X$ and $r>0.$ Find $y,z\in X$ such that $k_X(x_0,y) = k_X(x_0,z) = r$ and $L_f(x_0,r) = k_Y(f(x),f(y))$ and $\ell_f(x_0,r) = k_Y(f(x),f(z)).$ By assumption, $f$ is weakly $H$-quasisymmetric, where we can take $H = \eta(1)>0.$ Therefore, $k_X(x_0,y)\leq k_X(x_0,z)$ implies
    \begin{align*}
        k_Y(f(x_0),f(y)) & \leq H k_Y(f(x_0),f(z)).
    \end{align*}
    Hence,
    \begin{align}
        \frac{L_f(x_0,r)}{\ell_f(x_0,r)} & \leq H.\label{conversebound}
    \end{align}
    From (\ref{conversebound}), we deduce
    \begin{align*}
        H_f(x_0) & = \limsup_{r\rightarrow 0} \frac{L_f(x_0,r)}{\ell_f(x_0,r)} \leq H.
    \end{align*}
    As we chose $x_0$ arbitrarily,$\|H_f\|_\infty\leq H$ and so  $f$ is quasiconformal by Definition \ref{metricspacedefofqc}. We note that using the metric definition for quasiconformal maps is in fact equivalent to the analytic definition here since the Euclidean metric is locally equivalent to the quasihyperbolic metric. 
\end{proof}

We leave it open whether the results in section $3$ hold for other hyperbolic-type metrics such as Seittenranta's metric, the distance ratio metric, and Ferrand's metric. Refer to \cite{hariri2020conformally} for more on this.

\section{Function Theory on Uniformly K-QCH Domains}

 Here our objective is to uncover geometric properties of normal quasiregular mappings. Recall that Definition \ref{normalqrmapdef} generalizes the notion of a normal quasiregular map as originally described by the first author and Nicks in \cite{fletcher2024normal}. Our generalization required $X\subset S^n$ be as a uniformly quasiconformally homogeneous domain equipped with the quasihyperbolic metric. Throughout the proofs below, recall that $\sigma$ denotes the spherical distance function, and $(S^n, \sigma)$ is a metric space.

\begin{proof}[Proof of Theorem \ref{uniformconttheorem}]
    We first prove (i). Suppose $f$ is normal. By Theorem \ref{FNTheorem3.7}, and Proposition \ref{FNprop2.6} given a compact set $E\subset X,$ find a constant $L_E>0$ such that for every $x,y\in E$ and $A\in G,$ we have
    \begin{align}
        \sigma(f(A(x)),f(A(y))) & \leq L_Ek_X(x,y)^\alpha,\label{i1}
    \end{align}
    where we recall that $\alpha = M^\frac{1}{1-n}$ here.
    Next, fix $x_0\in X$ and find $r_0>0$ such that $B_X(x_0,r_0)$ is relatively compact in $X.$ Let $E$ be the compact set $\overline{B_X(x_0,r_0)}.$ Given $\varepsilon >0,$ choose
    \begin{align}
        \delta & < \min\left\{\frac{r_0}{2c},\left(\frac{r_0}{2c}\right)^{\frac{1}{\mu}}, \left(\frac{\varepsilon}{L_Ec^{\alpha}}\right)^\frac{1}{\alpha}, \left(\frac{\varepsilon}{L_Ec^{\alpha}}\right)^\frac{1}{\alpha\mu}\right\},\label{i2}
    \end{align}
    where $\mu = M^{\frac{1}{1-n}}$ and $c>0$ is the constant from Theorem \ref{GO}.
    If $x,y\in X$ with $k_X(x,y)<\delta,$ find $A\in G$ such that $A(x_0) = x.$ Set $y' = A^{-1}(y).$ Then, by Theorem \ref{GO},
    \begin{align}
        k_X(x_0,y') & = k_X(A^{-1}(x),A^{-1}(y)) \leq c\max\{k_X(x,y),k_X(x,y)^\mu\} < \frac{r_0}{2}.\label{i3}
    \end{align}
    Note that $\mu = M^\frac{1}{1-n} = \frac{1}{M^\frac{1}{n-1}}.$ Thus, $y'\in B_X(x_0,r_0)\subset E.$ Moreover, by (\ref{i1}), (\ref{i2}), and (\ref{i3}) we have the following
    \begin{align*}
        \sigma(f(x),f(y)) & = \sigma(f(A(x_0)),f(A(y')))\\
                          & \leq L_E k_X(x_0,y')^\alpha\\
                          & = L_Ek_X(A^{-1}(x),A^{-1}(y))^\alpha\\
                          & \leq L_E \left( c\max\{k_X(x,y),k_X(x,y)^\mu\}\right)^\alpha\\
                          & <\varepsilon.
    \end{align*}
    Thus, $f$ is uniformly continuous.
    
    For the converse, suppose $f$ is not normal. Hence $\mathcal{F}$ is not relatively compact in $C(X,S^n).$ By Theorem \ref{FNTheorem2.8}, it follows that $\mathcal{F}$ is not equicontinuous. This means there exists $\varepsilon>0$ and sequences $(x_m)_{m=1}^\infty,(y_m)_{m=1}^\infty\in X$ and $(A_m)_{m=1}^\infty\in G$ such that $k_X(x_m,y_m)\rightarrow 0$ as $m\rightarrow \infty$ but $\sigma(f(A_m(x_m)),f(A_m(y_m)))\geq \varepsilon.$ Again by Theorem \ref{GO} we have
    \begin{align*}
        0 & \leq k_X(A_m(x_m),A_m(y_m)) \leq c\max\{k_X(x_m,y_m), k_X(x_m,y_m)^\mu\}.
    \end{align*}
    
    Then, by the squeeze theorem, $k_X(A_m(x_m),A_m(y_m))\rightarrow 0$ as $m\rightarrow \infty.$ If we set $u_m = A_m(x_m)$ and $v_m = A_m(y_m)$ then we have $k_X(u_m,v_m)\rightarrow 0$ as $m\rightarrow \infty$ implies $\sigma(f(u_m),f(v_m))\geq \varepsilon.$ Hence, $f$ is not uniformly continuous.
    
    We now turn to (ii). If $f:X\rightarrow \mathbb{R}^n$ is normal then $\mathcal{F} = \{f(A(x))-f(A(x_0))\}$ is relatively compact in $\mathcal{Q}_{KM}(X,\mathbb{R}^n).$ We can then proceed as we did in (i) by using Theorem \ref{FNTheorem3.7}, Proposition \ref{FNprop2.6}, and Theorem \ref{GO}.
    
    For the converse, suppose $f:X\rightarrow \mathbb{R}^n$ is not normal. This implies $\mathcal{F}$ is not relatively compact in $\mathcal{Q}_{KM}(X,\mathbb{R}^n).$ Then, by Theorem \ref{FNTheorem2.8}, we can conclude that either $\mathcal{F}$ is not equicontinuous or $\mathcal{F}$ has an unbounded orbit. If $\mathcal{F}$ is not equicontinuous, then we can use the argument we used to prove the converse of (i). This follows since
    \begin{align*}
        \abs{f(A_m(x_m))-f(A_m(x_0)) -(f(A_m(y_m))-f(A_m(x_0)))} & = \abs{f(A_m(x_m))-f(A_m(y_m))}.
    \end{align*}
    
    What remains is to prove the case where $\mathcal{F}$ has an unbounded orbit. For sake of contradiction, suppose $\mathcal{F}$ has an unbounded orbit and that $f$ is uniformly continuous. Find $x_0,u\in X$ and $A_k\in G$ such that
    \begin{align}
        \abs{f(A_k(u))-f(A_k(x_0))} & \rightarrow \infty,\label{ii1}
    \end{align}
    as $k\rightarrow \infty.$ Let $\varepsilon = 1.$ Then, we can find $\delta >0$ such that $k_X(x,y)<\delta$ implies $\abs{f(x)-f(y)}<1.$ Let $T\geq 1$ such that $\delta<Tc,$ where $c$ is the constant from Theorem \ref{GO}. Then, we can find finitely many points $x_1,\ldots, x_m$ in $X$ with $x_m = u$ and $k_X(x_{i-1},x_i)<\left(\delta/Tc\right)^\frac{1}{\mu}<1$ for $i = 1,\ldots, m$ and where $\mu = M^\frac{1}{1-n}.$ Now, set $x_i^k = A_k(x_i)$ for $i = 0,\ldots, m$ and $k\in \mathbb{N}.$ Then, by Theorem \ref{GO}, notice
    \begin{align*}
        k_X(x_{i-1}^k,x_i^k) & = k_X(A(x_{i-1}),A(x_i))\\
                             & \leq c\max\{k_X(x_{i-1},x_i),k_X(x_{i-1},x_i)^\mu\}\\
                             & = ck_X(x_{i-1},x_i)^\mu\\
                             & < c\left(\frac{\delta}{Tc}\right)\\
                             & \leq \delta,
    \end{align*}
    for all $k\in \mathbb{N}.$ By the triangle inequality and uniform continuity, we have for all $k\in \mathbb{N}$ that
    \begin{align*}
        \abs{f(A_k(u))-f(A_k(x_0))} & \leq \sum_{i=1}^m\abs{f(A_k(x_i))-f(A_k(x_{i-1}))} = \sum_{i=1}^m \abs{f(x_i^k)-f(x_{i-1}^k)}<m.
    \end{align*}
    However, this contradicts (\ref{ii1}) and so we conclude that $f$ cannot be uniformly continuous, as desired.
    
\end{proof}

The next result we prove tells us that $f$ being a normal map is equivalent to the map having global H\"{o}lder behavior. This result is important, as it provides us insight into geometric properties of such normal maps.

\begin{proof}[Proof of Theorem \ref{holderSn}]
    We first assume $f$ is normal. Fix $x_0\in X$ and find $r>0$ so that $B_X(x_0,2r)$ is relatively compact in $X.$ Let $E = \overline{B_X(x_0,2r)}.$ By Theorem \ref{FNTheorem3.7} and Proposition \ref{FNprop2.6}, find $L_E>0$ such that 
    \begin{align}
        \sigma(f(A(x)),f(A(y))) & \leq L_Ek_X(x,y)^\alpha\label{4.2.2}
    \end{align}
    for all $x,y\in E$ and all $A\in G.$ Now, find $T\geq 1$ such that the following inequality is satisfied
    \begin{align*}
        r & < Tc,
    \end{align*}
    where $c>0$ is the constant from Theorem \ref{GO}. Set
    \begin{align*}
        R & = \left(\frac{r}{Tc}\right)^\frac{1}{\mu}< 1,
    \end{align*}
    where $\mu = M^\frac{1}{1-n}.$ Suppose first that $x,y\in X$ with $k_X(x,y)\geq R.$ We assumed $M,K\geq 1$ and $R<1$ so  
    \begin{align*}
        k_X(x,y)^\beta & \geq R^\beta\\
        k_X(x,y)^\beta &\geq \left( \frac{r}{Tc}\right)^\frac{\beta}{\mu}
    \end{align*}
    and
    \begin{align*}
        \pi\left( \frac{r}{Tc}\right)^{-\frac{\beta}{\mu}}k_X(x,y)^\beta & \geq \pi.
    \end{align*}
    Since spherical distances on $S^n$ are bounded above by $\pi,$ it follows that
    \begin{align*}
        \sigma(f(x),f(y)) & \leq \pi \leq \pi\left( \frac{r}{Tc}\right)^{-\frac{\beta}{\mu}}k_X(x,y)^\beta.
    \end{align*}
    Thus, (\ref{4.2.1}) holds trivially. Next, consider $x,y\in X$ with $k_X(x,y)<R.$ Find $A\in G$ such that $A(x_0) = y$ and set $x_1 = A^{-1}(x).$ Then, since $k_x(x,y)<1$ and $T\geq 1,$ notice
    \begin{align*}
        k_X(x_0,x_1) & = k_X(A^{-1}(x),A^{-1}(y))\\
                     & \leq c\max\{k_X(x,y),k_X(x,y)^\mu\}\\
                     & = ck_X(x,y)^\mu\\
                     & < cR^\mu\\
                     & = c\left(\frac{r}{Tc}\right)\\
                     & = \frac{r}{T}\\
                     & < r.
    \end{align*}
    Thus, $x_0,x_1\in E.$ Then, using (\ref{4.2.2}), Theorem \ref{GO}, and that $k_X(x,y)<1,$ we have
    \begin{align*}
        \sigma(f(x),f(y)) & = \sigma(f(A(x_1)),f(A(x_0)))\\
                          & \leq L_Ek_X(x_0,x_1)^\alpha\\
                          & = L_Ek_X(A^{-1}(x),A^{-1}(y))^\alpha\\
                          & \leq L_Ec^\alpha\max\{k_X(x,y),k_X(x,y)^\mu\}^\alpha\\
                          & = L_Ec^\alpha k_X(x,y)^{\mu\alpha}
    \end{align*}
    Note that $\beta = \alpha\mu.$ Then, setting $C_0 = L_Ec^\alpha$ we see that $f$ is globally $\beta-$H\"{o}lder.\\
    Conversely, suppose (\ref{4.2.1}) holds for all $x,y\in X.$ Let $x_0\in X, r>0$ and $x,y\in B_X(x_0,r).$ If $A\in G,$ set $u = A(x)$ and $v = A(y).$ Then, using (\ref{4.2.2}) and Theorem \ref{GO},
    \begin{align*}
        \sigma(f(A(x)),f(A(y))) & = \sigma(f(u),f(v))\\
                                & \leq C_0k_X(u,v)^\beta\\
                                & = C_0k_X(A(x),A(y))^\beta\\
                                & \leq C_0 c^\beta \max\{k_X(x,y),k_X(x,y)^\mu\}^\beta.
    \end{align*}
    Thus, if $\omega(t) = \max\{t,t^\mu\}^\beta,$ then $\mathcal{F} = \{f\circ A:A\in G\}$ is locally uniformly $\omega-$continuous. Now, for any $x'\in X,$ $\overline{\mathcal{F}(x')}\subset S^n$ is a closed set. Since $S^n$ is compact, $\overline{\mathcal{F}(x')}$ is compact. Thus, $\mathcal{F}(x')$ is relatively compact in $S^n.$ Hence, by Theorem \ref{FNTheorem2.9}, $\mathcal{F}$ is normal, and so $f$ is normal.
\end{proof}

Next, we find a result similar to the previous theorem, but for when our range is $\mathbb{R}^n$ rather than $S^n.$ Changing to $\mathbb{R}^n$ equipped with Euclidean metric tells us that $f$ normal is equivalent to saying $f$ has globally Lipschitz behavior, but locally H\"{o}lder behavior. Although we still gain geometric insight, it is noticeably different than the previous theorem.

\begin{proof}[Proof of Theorem \ref{holderRn}]
    We first assume $f$ is normal. Fix $x_0\in X$ and find $r>0$ so that $B_X(x_0,2r)$ is relatively compact in $X.$ Let $E = \overline{B_X(x_0,r)}.$ By Corollary \ref{FNCorollary3.10}, find $L_E>0$ such that 
    \begin{align}
        \abs{f(A(x'))-f(A(y')))} & \leq L_Ek_X(x',y')^\alpha\label{4.5.1}
    \end{align}
    for all $x',y'\in E$ and all $A\in G.$ Now, find $T\geq 1$ such that
    \begin{align*}
        r & < Tc,
    \end{align*}
    where $c>0$ is the constant from Theorem \ref{GO}. Set
    \begin{align*}
        R & = \left(\frac{r}{Tc}\right)^\frac{1}{\mu}< 1,
    \end{align*}
    where $\mu = M^\frac{1}{1-n}.$ Now, suppose $x,y\in X$ with $k_X(x,y)\leq R.$ Find $A\in G$ such that $x_0 = A(x)$ and $y' = A(y).$ Then, using Theorem \ref{GO} and that $k_X(x,y)<1,$ notice
    \begin{align*}
        k_X(x_0,y') & = k_X(A(x),A(y))\\
                    & \leq c\max\{k_X(x,y),k_X(x,y)^\mu\}\\
                    & = ck_X(x,y)^\mu\\
                    & < cR^\mu\\
                    & = \frac{r}{T}\\
                    & < r.
    \end{align*}
    Thus, $y'\in E.$ Then, using (\ref{4.5.1}), Theorem \ref{GO}, and that $k_X(x,y)<1,$ we have
    \begin{align*}
        \abs{f(x)-f(y)} & = \abs{f(A^{-1}(x_0))-f(A^{-1}(y'))}\\
                        & \leq L_Ek_X(x_0,y')^\alpha\\
                        & = L_E k_X(A(x),A(y))^\alpha\\
                        & \leq L_Ec^\alpha\max\{k_X(x,y),k_X(x,y)^\mu\}^\alpha\\
                        & = L_Ec^\alpha k_X(x,y)^{\alpha\mu}\\
                        & = L_Ec^\alpha k_X(x,y)^\beta.
    \end{align*}
    Now, suppose $x,y\in X$ with $k_X(x,y)>R.$ Further, suppose $m\in \mathbb{N}$ such that
    \begin{align}
        k_X(x,y) \in (mR/2,(m+1)R/2].\label{4.5.2}
    \end{align}
    Since $k_X(x,y) = \inf_\gamma \int_\gamma \frac{1}{d(z,\partial X)}\abs{dz},$ we can find a sequence of points $z_0,z_1,\ldots, z_m$ with $z_0 = x,z_m = y, k_X(z_i,z_{i+1}) = R/2$ for $i = 0,\ldots, m-2$ and $k_X(z_{m-1},z_m)<R.$ Then, from the first case we discussed and (\ref{4.5.2}), we have
    \begin{align*}
        \abs{f(x)-f(y)} & \leq \sum_{i=0}^{m-1}\abs{f(z_i)-f(z_{i+1})}\\
                        & \leq \sum_{i=0}^{m-1}L_Ec^\alpha k_X(z_i,z_{i+1})^{\beta}\\
                        & <L_Ec^\alpha(m-1)\frac{R^\beta}{2^\beta} + L_Ec^\alpha R^\beta\\
                        & < L_Ec^\alpha R^\beta m\\
                        & \leq 2L_Ec^\alpha R^{\beta-1}k_X(x,y).
    \end{align*}
    It then follows that (\ref{4.5}) holds globally with $C_0 = 2L_Ec^\alpha R^{\beta-1}.$
    
    Conversely, suppose that (\ref{4.5}) holds for all $x,y\in X.$ Let $x_1\in X,$ choose $0<r<1/2,$ and suppose $x,y\in B_X(x_1,r).$ If $A\in G$ then by Theorem \ref{GO},
    \begin{align*}
        \abs{f(A(x)),f(A(y))} & \leq C_0 k_X(A(x),A(y))^{\mu\alpha}\\
                              & \leq C_0c^{\beta}\max\{k_X(x,y),k_X(x,y)^\mu\}^{\alpha\mu}\\
                              & = C_0c^\beta k_X(x,y)^{\mu^2\alpha}.
    \end{align*}
    
    Thus, for $x_0\in X,\mathcal{F} = \{f(A(x))-f(A(x_0))|A\in G\}$ is locally uniformly $\omega$-continuous with $\omega(t) = t^{\mu^2\alpha}.$ Furthermore, notice that $\mathcal{F}(x_0) = \{0\}$ which is relatively compact in $\mathbb{R}^n.$ Thus, $\mathcal{F}$ is relatively compact in $C(X,\mathbb{R}^n)$ by Theorem \ref{FNTheorem2.9}. Hence, $\mathcal{F}$ is normal and so $f$ is normal.
\end{proof}

\bibliographystyle{plain}  

\vspace{.5 in}
\noindent Department of Mathematical Sciences, Northern Illinois University, DeKalb, IL 60115-2888,USA,\\
\textit{Email Address:} \texttt{fletcher@math.niu.edu}
\vspace{.1 in}

\noindent Department of Mathematical Sciences, Northern Illinois University, DeKalb, IL 60115-2888,USA,\\
\textit{Email Address:} \texttt{ahahn2813@gmail.com}

\end{document}